\documentclass[bj,preprint]{imsart}
\usepackage[utf8]{inputenc}
\usepackage[T1]{fontenc}
\usepackage[english]{babel}

\arxiv{arXiv:1703.09433}

\usepackage{amsmath,xcolor,color,dsfont,verbatim}
\usepackage{amsfonts}
\usepackage{amssymb,amsthm}

\colorlet{darkgreen}{green!50!black}

\usepackage{hyperref} 
\hypersetup{	
colorlinks=true,  
breaklinks=true,  
urlcolor= blue, %green, 
linkcolor= blue,	
citecolor=darkgreen,	
pdftitle={Integral expression for the stationary distribution of reflected Brownian motion in a wedge}, 
pdfauthor={Sandro Franceschi and Kilian Raschel},
}

\newtheorem{thm}{Theorem}
\newtheorem{prop}[thm]{Proposition}
\newtheorem{cor}[thm]{Corollary}

\newtheorem{lem}[thm]{Lemma}
\newtheorem{rem}[thm]{Remark}

\renewcommand{\geq}{\geqslant}
\renewcommand{\leq}{\leqslant}

\usepackage{epsfig}

\usepackage{geometry}
\usepackage{enumitem}
\usepackage{mathtools}
\usepackage{eqnarray,cases}
\usepackage{multirow}

\usepackage{graphicx}
\usepackage{subcaption}

\usepackage{indentfirst}

\usepackage{pstricks-add}
\usepackage[numbers]{natbib}

\newcommand{\citeep}[1]{\citeeauthor*{#1} (\citeeyear{#1}) \citeep{#1}}

\usepackage{float}

\DeclareMathOperator{\sgn}{sgn}

\newcommand{\G}{\mathcal G_\mathcal R}
\newcommand{\R}{\mathcal{R}}
\let\phi=\varphi
\let\epsilon=\varepsilon

\begin{document}

\begin{frontmatter}

\title{Integral expression for the stationary distribution of reflected Brownian motion in a wedge}
\runtitle{Integral expression for the stationary distribution of reflected Brownian motion in a wedge}
%[Stationary distribution of reflected Brownian motion in a wedge]

\begin{aug}
  \author{\fnms{Sandro}  \snm{Franceschi}\thanksref{a,e1}\ead[label=e1,mark]{sandro.franceschi@sorbonne-universite.fr}}
  \and
  \author{\fnms{Kilian} \snm{Raschel}\thanksref{b,e2}\ead[label=e2,mark]{raschel@math.cnrs.fr}}

  \runauthor{S. Franceschi and K. Raschel}

  \affiliation{Some University and Another University}

  \address[a]{Laboratoire de Probabilit\'es, Statistique et Mod\'elisation, Sorbonne Universit\'e,
        4 Place Jussieu, 75252 Paris Cedex~05, France. \printead{e1}}

  \address[b]{CNRS and Institut Denis Poisson, Universit\'e de Tours and Universit\'e d'Orl\'eans, Parc de Grandmont, 37200 Tours, France. \printead{e2}}

\end{aug}

\begin{abstract}\
For Brownian motion in a (two-dimensional) wedge with negative drift and oblique reflection on the axes, we derive an explicit formula for the Laplace transform of its stationary distribution (when it exists), in terms of Cauchy integrals and generalized Chebyshev polynomials. To that purpose we solve a Carleman-type boundary value problem on a hyperbola, satisfied by the Laplace transforms of the boundary stationary distribution.
\end{abstract}


\end{frontmatter}

%\subjclass{Primary 05A15; Secondary 30F10, 30D05}

%\footnote{Version of \today}

\begin{figure}[ht]
\centering
\vspace{-5mm}
\includegraphics[scale=0.75]{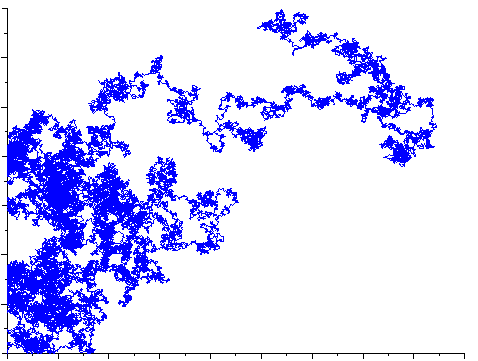}
\caption{An example of path of reflected Brownian motion in the quadrant}
%\label{fig:drift_reflection}
\end{figure}

\section{Introduction and main results}
\label{sec:introduction}

Since its introduction in the eighties by Harrison, Reiman, Varadhan and Williams \cite{HaRe-81,HaRe-81b,varadhan_brownian_1985,Williams-85b,Williams-85}, reflected Brownian motion in the quarter plane has received a lot of attention from the probabilistic community. However, and surprisingly, finding a general explicit expression of the stationary distribution has been left as an open problem. The present paper solves this problem in a complete and unified way.

\subsection*{Reflected Brownian motion in two-dimensional cones}
The semimartingale reflected Brownian motion with drift in the quarter plane $\mathbb{R}_+^2 :=[0,\infty)^2$ (or equivalently in arbitrary convex wedges, by performing a simple linear transformation, cf.\ Appendix \ref{app:BM_cones}) can be written as
\begin{equation}
\label{eq:RBMQP}
     Z(t)=Z_0 + B(t) + \mu\cdot t + R\cdot L(t),\qquad \forall t\geq 0,
\end{equation}
where  
\begin{itemize}
     \item $Z_0$ is any initial point in the quadrant,
     \item $B$ is a Brownian motion with covariance $\Sigma=\small\left(  \begin{array}{ll} \sigma_{11} & \sigma_{12} \\ \sigma_{12} & \sigma_{22} \end{array} \right)$ starting from the origin,
     \item $\mu= \small\left(  \begin{array}{l} \mu_1 \\  \mu_2  \end{array} \right)$ denotes the interior drift,
     \item  $R=(R^1,R^2) =\small \left(  \begin{array}{cc} r_{11} & r_{12} \\ r_{21} & r_{22} \end{array} \right)$ is the reflection matrix,
     \item $L= \small\left(  \begin{array}{l} L^1 \\  L^2\end{array} \right)$ is the local time. 
\end{itemize}
For $i\in\{1,2\}$, the local time $L^i(t)$ is a continuous non-decreasing process starting from $0$ (i.e., $L^i(0)=0$), increasing only at time $t$ such that $Z^i(t)=0$, viz., $\int_{0}^t \mathds{1}_{\{Z^i(s) \ne 0 \}} \mathrm{d} L^i(s)=0$, for all $t\geq 0$. The columns $R^1$ and $R^2$ represent the directions in which the Brownian motion is pushed when the axes are reached, see Figure~\ref{fig:drift_reflection}.

The reflected Brownian motion $(Z(t))_{t\geq0}$ associated with $(\Sigma, \mu, R)$ is well defined \cite{varadhan_brownian_1985,Williams-85}, and is a fundamental stochastic process. 
This process has been extensively explored, and its multidimensional
version (a semimartingale reflected Brownian motion in the positive orthant $\mathbb{R}_+^d$, as
well as in convex polyhedrons) as well.
It has been studied in depth, 
with focuses on its definition and semimartingale properties \cite{varadhan_brownian_1985,Williams-85b,DaKu-94,williams_semimartingale_1995}, its recurrence or transience \cite{Williams-85b,dai_steady-state_1990,hobson_recurrence_1993,Ch-96,BrDaHa-10,Br-11,DaHa-12}, the possible particular (e.g., product) form of its stationary distribution \cite{harrison_multidimensional_1987,DiMo-09,OCOr-14}, its Lyapunov functions \cite{DuWi-94}, its links with other stochastic processes \cite{lega-87,Du-04,Lep}, its use to approximate large queuing networks \cite{foddy_1984,Baccelli_Fayolle_1987,harrison_brownian_1987,KeWh-96,KeRa-12}, the asymptotics of its stationary distribution \cite{harrison_reflected_2009,dai_reflecting_2011,franceschi_asymptotic_2016,Sa-+2}, numerical methods to compute the stationary distribution \cite{dai_steady-state_1990,dai_reflected_1992}, links with complex analysis \cite{foddy_1984,Baccelli_Fayolle_1987,BuChMaRa_2015,franceschi_tuttes_2016}, comparison and coupling techniques \cite{Sa-+1,Sa-15}, etc. %We also refer to \cite{BuChMaRa_2015} for the analysis of reflected Brownian motion in bounded domains by complex analysis techniques. 

The main contribution of the present paper is to find an exact expression for the stationary distribution (via its Laplace transforms, to be introduced in \eqref{eq:Laplace_transform_interior} and \eqref{eq:Laplace_transform_boundary}), thanks to the theory of boundary value problems (BVPs), see our Theorem \ref{thm:main}. This is one of the first attempts to apply boundary value techniques to diffusions in the quadrant, after \cite{Foschini} (under the symmetry conditions $\mu_1 = \mu_2$, $\sigma_{11}=\sigma_{22}$, and symmetric reflection vectors in \eqref{eq:RBMQP}), \cite{foddy_1984} (which concerns very specific cases of the covariance matrix, essentially the identity), \cite{Baccelli_Fayolle_1987} (on diffusions with a special behavior on the boundary), \cite{franceschi_tuttes_2016} (orthogonal reflections, solved by Tutte's invariant approach \cite{Tutte-95,bernardi_counting_2015}); see also the introduction of \cite{dai_reflecting_2011}, where the authors allude to the possibility of such an approach. \textcolor{black}{The application of BVP techniques to discrete models has a longer history, see \cite{fayolle_random_1999,kurkova_malyshevs_2003,bernardi_counting_2015} and references therein.}

%\begin{figure}[t]
%\centering
%\vspace{-5mm}
%\includegraphics[scale=0.5]{rebondrec1.pdf}
%\includegraphics[scale=0.5]{rebondrec2.pdf}
%\includegraphics[scale=0.5]{rebondrec3.pdf}
%\vspace{-5mm}
%\caption{Drift $\mu$ and reflection vectors $R^1$ and $R^2$}
%\label{fig:drift_reflection}
%\end{figure}

\begin{figure}[t]
    \centering
    \begin{subfigure}{0.3\textwidth}
        \includegraphics[scale=0.5]{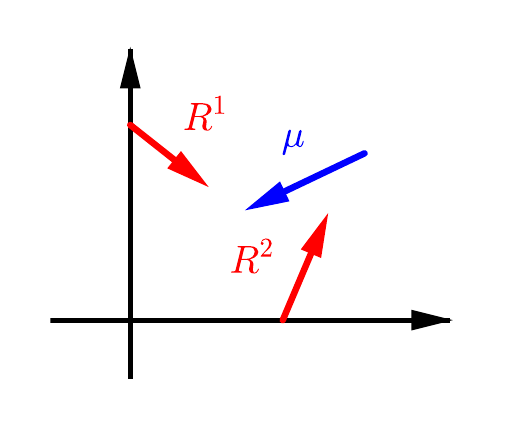}
        \caption{$\mu_1<0, \ \mu_2<0$}
        \label{fig:gull}
    \end{subfigure}
    \begin{subfigure}{0.3\textwidth}
        \includegraphics[scale=0.5]{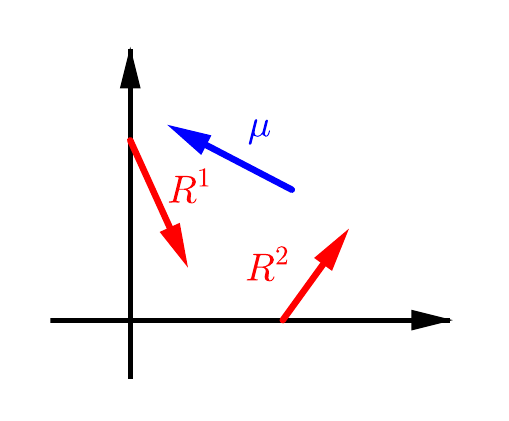}
        \caption{$\mu_1<0, \ \mu_2\geq0$}
        \label{fig:tiger}
    \end{subfigure}
    \begin{subfigure}{0.3\textwidth}
        \includegraphics[scale=0.5]{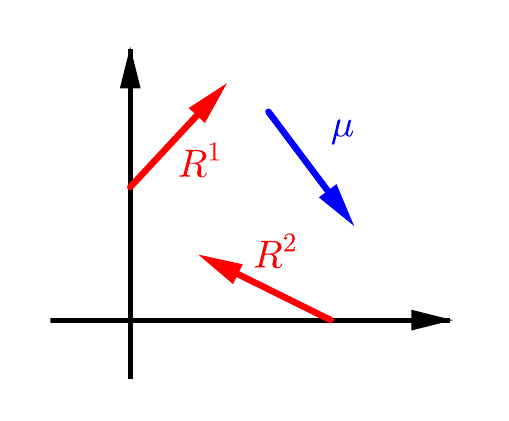}
        \caption{$\mu_1\geq0, \ \mu_2<0$}
        \label{fig:mouse}
    \end{subfigure}
    \caption{Drift $\mu$ and reflection vectors $R^1$ and $R^2$}
\label{fig:drift_reflection}
\end{figure}

\subsection*{Laplace transforms and functional equation}

The reflected Brownian motion defined in \eqref{eq:RBMQP} exists if and only if
\begin{equation*}
     \{r_{11}>0,\quad r_{22}>0,\quad \det R >0\} \quad \text{or} \quad \{r_{11}>0,\quad r_{22}>0,\quad  r_{12}>0,\quad r_{21}>0\}.
\end{equation*}
This condition is equivalent for $R$ to be completely-$\mathcal{S}$, which actually is a necessary and sufficient condition to the existence of reflected Brownian motion in arbitrary dimension, see \cite{taylor_existence_1993,reiman_boundary_1988}.

As for the stationary distribution, it exists if and only if
\begin{equation}
\label{eq:stationary_distribution_CNS}
     r_{11} > 0, \quad r_{22} > 0, \quad r_{11} r_{22} - r_{12} r_{21} > 0,\quad r_{22} \mu_1 - r_{12}  \mu_2 < 0, \quad r_{11} \mu_2 - r_{21}  \mu_1 < 0,
\end{equation}
see \cite{harrison_brownian_1987,harrison_reflected_2009}, and in that case it is absolutely continuous w.r.t.\ the Lebesgue measure \cite{dai_steady-state_1990,dai_reflected_1992,harrison_brownian_1987}, with density denoted by $\pi(x)=\pi(x_1,x_2)$. \textcolor{black}{See Figure \ref{fig:drift_reflection} for an example of parameters satisfying to \eqref{eq:stationary_distribution_CNS}; there are three different cases, according to the sign of the drift coordinates (having two non-negative coordinates is obviously incompatible with \eqref{eq:stationary_distribution_CNS}).}
Assumption \eqref{eq:stationary_distribution_CNS} implies in particular that $R$ is invertible and $R^{-1}\mu<0$, which turns out to be a necessary condition for the existence of the stationary distribution in any dimension, see \cite{harrison_brownian_1987}. From now, we will assume that \eqref{eq:stationary_distribution_CNS} is satisfied.
Let the Laplace transform of $\pi$ be defined by
\begin{equation}
\label{eq:Laplace_transform_interior}
     \varphi (\theta) = \mathbb{E}_{\pi} [\exp{ (\theta \cdot  Z)}] = \iint_{{\mathbb R}_+^2} \exp{( \theta \cdot x )} \pi(x) \mathrm{d} x.
\end{equation}
Furthermore we define two finite boundary measures $\nu_1$ and $\nu_2$ such that, for $A\subset\mathbb{R}_+$,
\begin{equation*}
     \nu_1 (A) = \mathbb{E}_{\pi} \bigg[ \int_0^1 \mathds{1}_{\{Z(t) \in \{0\}\times A\}} \mathrm{d}L^1 (t)\bigg],\qquad 
     \nu_2 (A) = \mathbb{E}_{\pi} \bigg[ \int_0^1 \mathds{1}_{\{Z(t) \in A \times\{0\} \}} \mathrm{d}L^2 (t)\bigg].
\end{equation*}     
The $\nu_i$ have their supports on the axes and may be viewed as boundary invariant measures. They are continuous w.r.t.\ the Lebesgue measure, see \cite{harrison_multidimensional_1987}. We define their Laplace transform by% (a priori for values of the argument with non-positive real parts)
\begin{equation}
\label{eq:Laplace_transform_boundary}
     \varphi_1 (\theta_2) =\int_{{\mathbb R}_+} \exp({\theta_2 x_2}) \nu_1(x_2) \mathrm{d} x_2,\qquad
       \varphi_2 (\theta_1) =\int_{{\mathbb R}_+} \exp({\theta_1 x_1}) \nu_2(x_1) \mathrm{d} x_1.
\end{equation}
The following functional equation relates the Laplace transforms:
\begin{equation}  
\label{eq:functional_equation}
     -\gamma (\theta) \varphi (\theta) =\gamma_1 (\theta) \varphi_1 (\theta_2) + \gamma_2 (\theta) \varphi_2 (\theta_1),
\end{equation}
where we have noted
\begin{equation}
\label{eq:def_gamma_gamma1_gamma2}
  \begin{cases}
     \hspace{1.5mm}\gamma (\theta)= \frac{1}{2}  \textcolor{black}{(\theta \cdot  \Sigma \theta) +  \theta \cdot  \mu}  =\frac{1}{2}(\sigma_{11} \theta_1^2+2\sigma_{12}\theta_1\theta_2 + \sigma_{22} \theta_2^2 )
+
\mu_1\theta_1+\mu_2\theta_2,  \\
     \gamma_1 (\theta)=  \textcolor{black}{R^1 \cdot  \theta} =r_{11} \theta_1 + r_{21} \theta_2, \\
     \gamma_2 (\theta)= \textcolor{black}{R^2 \cdot  \theta} =r_{12} \theta_1 + r_{22} \theta_2.
  \end{cases}
\end{equation}
The Laplace transforms \eqref{eq:Laplace_transform_interior} and \eqref{eq:Laplace_transform_boundary} (resp.\ Equation \eqref{eq:functional_equation}) exist (resp.\ holds) at least for values of $\theta=(\theta_1, \theta_2)$ with ${\Re}\,\theta_1\leq 0$ and ${\Re}\, \theta_2\leq 0$. To prove the functional equation \eqref{eq:functional_equation}, the main idea is to use an identity called basic adjoint relationship (BAR); see \cite[Section 2.1]{franceschi_tuttes_2016} and \cite{dai_reflecting_2011,foddy_1984} for details.

\subsection*{Kernel and associated quantities}

In this paragraph we introduce necessary notation to state our main results. By definition, the kernel of \eqref{eq:functional_equation} is the second degree polynomial $\gamma$. With this terminology, \eqref{eq:functional_equation} is sometimes referred to as a kernel equation.
The equality $\gamma(\theta_1, \theta_2) = 0$ with $\theta_1,\theta_2\in\mathbb C$ defines algebraic functions $\Theta_1^\pm(\theta_2)$ and $\Theta_2^\pm(\theta_1)$ by 
\begin{equation*}
     \gamma(\Theta_1^\pm(\theta_2), \theta_2)=\gamma(\theta_1, \Theta_2^\pm(\theta_1))= 0.
\end{equation*}
Solving these equations readily yields
\begin{equation}
\label{eq:definition_Theta_pm}
     \left\{\begin{array}{l}
     \Theta_1^\pm(\theta_2)=\dfrac{-(\sigma_{12}\theta_2+\mu_1)\pm\sqrt{\theta_2^2(\sigma_{12}^2-\sigma_{11}\sigma_{22})+2\theta_2(\mu_1\sigma_{12}-\mu_2\sigma_{11})+\mu_1^2}}{\sigma_{11}},\smallskip\smallskip\\
     \Theta_2^\pm(\theta_1)=\dfrac{-(\sigma_{12}\theta_1+\mu_2)\pm\sqrt{\theta_1^2(\sigma_{12}^2-\sigma_{11}\sigma_{22})+2\theta_1(\mu_2\sigma_{12}-\mu_1\sigma_{22})+\mu_2^2}}{\sigma_{22}}.\end{array}\right.
\end{equation}
The polynomials under the square roots in \eqref{eq:definition_Theta_pm} have two zeros (called branch points), real and of opposite signs. They are denoted by $\theta_2^{\pm}$ and $\theta_1^\pm$, respectively:
\begin{equation}
\label{eq:definition_theta_pm}
\left\{\begin{array}{l}
     \displaystyle\theta_2^\pm = \frac{(\mu_1\sigma_{12}-\mu_2\sigma_{11}) \pm \sqrt{(\mu_1\sigma_{12}-\mu_2\sigma_{11})^2 +\mu_1^2 \det{\Sigma}}}{\det\Sigma},\smallskip\smallskip\\
     \displaystyle\theta_1^\pm= \frac{(\mu_2\sigma_{12}-\mu_1\sigma_{22}) \pm \sqrt{(\mu_2\sigma_{12}-\mu_1\sigma_{22})^2 +\mu_2^2 \det{\Sigma}}}{\det\Sigma}.\end{array}\right.
\end{equation}
\textcolor{black}{The algebraic functions $\Theta_1^\pm (\theta_2)$ are meromorphic on the cut plane $\mathbb{C}\setminus ((-\infty,\theta_2^-]\cup[\theta_2^+,\infty))$. Similarly, $\Theta_2^\pm (\theta_1)$ are meromorphic on $\mathbb{C}\setminus ((-\infty,\theta_1^-]\cup[\theta_1^+,\infty))$.}

Our next important definition is the curve
     \begin{equation*}
     \R =\{\theta_2\in\mathbb C : \gamma(\theta_1,\theta_2)=0 \text{ and } \theta_1\in(-\infty,\theta_1^-)\}.
\end{equation*}
As it will be seen in Lemma \ref{lem:Lemma9BaFa87} (see also Figure \ref{fig:BVP_theta}), $\R$ is a \textcolor{black}{branch of} hyperbola.
We denote by 
\begin{equation}
\label{eq:curve_definition_negative}
     \R^-=\{\theta_2 \in \R : \Im\,\theta_2 \leq 0 \}
\end{equation} 
the negative imaginary part of $\R$ oriented from the vertex to infinity.

We further define the function 
\begin{equation}
\label{eq:definition_w}
     w(\theta_2)=T_{\frac{\pi}{\beta}}\left(-\frac{2\theta_2-(\theta_2^+ +\theta_2^-)}{\theta_2^+ -\theta_2^-}\right),
\end{equation}
where 
\begin{equation}
\label{eq:definition_beta}
\beta= \arccos {-\frac{\sigma_{12}}{\sqrt{\sigma_{11}\sigma_{22}}}}
\end{equation}
and for $a\geq0$, $T_a$ is the so-called generalized Chebyshev polynomial
\begin{equation}
\label{eq:Chebyshev_polynomial}
     T_a(x)  =\cos (a\arccos x)=\frac{1}{2} \Big\{\big(x+\sqrt{x^2-1}\big)^a+\big(x-\sqrt{x^2-1}\big)^a\Big\}.
\end{equation}
The function $w$ plays a special role regarding the curve $\R$, as for $\theta_2\in\R$ it satisfies $w(\theta_2)=w(\overline{\theta_2})$, see Lemma \ref{lem:conformal_gluing}. (Here and throughout, $\overline{\theta_2}$ denotes the complex conjugate number of $\theta_2$.)

Finally, let $G$ be the function whose expression is
\begin{equation}
\label{eq:definition_G}
     G(\theta_2)=\frac{\gamma_1}{\gamma_2}(\Theta_1^-(\theta_2),\theta_2)\frac{\gamma_2}{\gamma_1}(\Theta_1^-(\overline{\theta_2}),\overline{\theta_2}).
\end{equation}

\subsection*{Main results}
Our main result can be stated as follows.
\begin{thm} 
\label{thm:main}
Under the assumption \eqref{eq:stationary_distribution_CNS}, the Laplace transform $\phi_1$ in \eqref{eq:Laplace_transform_boundary} is equal to
\begin{multline}
\label{eq:main_formula_with_constants}
\phi_1(\theta_2)
=\\\nu_1(\mathbb{R}_+)  \left( \frac{w(0)-w(p)}{w(\theta_2)-w(p)} \right)^{-\chi} \exp\bigg\{\frac{1}{2i\pi} \int_{\R^-} \log G(\theta) \left[ \frac{w'(\theta)}{w(\theta)-w(\theta_2)}
- \frac{w'(\theta)}{w(\theta)-w(0)}
\right]
\mathrm{d}\theta\bigg\},
\end{multline}
where
\begin{itemize}
     \item $w$, $G$ and $\R^-$ are defined in \eqref{eq:definition_w},  \eqref{eq:definition_G} and \eqref{eq:curve_definition_negative}, respectively,\smallskip
     \item $\displaystyle \nu_1(\mathbb{R}_+)=\frac{r_{12}\mu_2-r_{22}\mu_1 }{\det R}$,
     \item the index $\chi$ is given by $\chi=\left\{\begin{array}{rl}
     0 & \displaystyle\text{if } \gamma_1(\theta_1^-,\Theta_2^-(\theta_1^-))\leqslant 0,\smallskip\\
      -1 & \displaystyle\text{if } \gamma_1(\theta_1^-,\Theta_2^-(\theta_1^-))>0,
      \end{array}\right.$
      \item $\displaystyle p=\frac{2r_{11}(\mu_1r_{21}-\mu_2r_{11})}{r_{11}^2\sigma_{22}-2r_{11}r_{21}\sigma_{12}+r_{21}^2\sigma_{11}}$,\smallskip
      \item to define the function $\log G(\theta)$ on $\R^-$, we use the determination of the logarithm taking a value in $i\cdot(-\pi,\pi]$ at the vertex of $\R^-$ and varying continuously over the curve $\R^-$.
\end{itemize}
The function $\phi_2(\theta_1)$ equals $\phi_1(\theta_1)$ in \eqref{eq:main_formula_with_constants} after the change of parameters 
\begin{equation*}
     \sigma_{11}\leftrightarrow\sigma_{22},\quad \mu_{1}\leftrightarrow\mu_{2},\quad r_{11}\leftrightarrow r_{22},\quad r_{12}\leftrightarrow r_{21}. 
\end{equation*}
The functional equation~\eqref{eq:functional_equation} finally gives an explicit formula for the bivariate Laplace transform~$\phi$.
\end{thm}

Let us now give some comments around Theorem~\ref{thm:main}.
\begin{itemize}
     \item Theorem \ref{thm:main} completely generalizes the results of \cite{Foschini} (with symmetry conditions), \cite{foddy_1984} (with the identity covariance matrix $\Sigma$) and \cite{franceschi_tuttes_2016} (orthogonal reflections on the axes). It offers the first explicit expression of the Laplace transforms, covering all the range of parameters $(\Sigma,\mu,R)$ satisfying to \eqref{eq:stationary_distribution_CNS}, thereby solving an old open problem.
     \item We obtain three corollaries of Theorem \ref{thm:main}, each of those corresponds to an already known result: we first compute the asymptotics of the boundary densities (Section \ref{sec:asymptotic_results}, initially obtained in \cite{DaMi-13}); second we derive the product form expression of the density in the famous skew-symmetric case (Section \ref{subsec:skew}, result originally proved in \cite{harrison_multidimensional_1987}); finally we compute the expression of the Laplace transforms in the case of orthogonal reflections (Section \ref{sec:orthogonal_reflection}, see \cite{franceschi_tuttes_2016} for the first derivation).
     \item It is worth remarking that the expression \eqref{eq:main_formula_with_constants} is intrinsically non-continuous in terms of the parameters: the index $\chi$ can indeed take two different values (namely, $0$ and $1$). For this reason, \eqref{eq:main_formula_with_constants} actually contains two different formulas. See Remarks \ref{rem:different_cases_singularities} and \ref{rem:different_cases_asymptotic} for further related comments.
     \item The paper \cite{franceschi_asymptotic_2016} obtains the exact asymptotics of the stationary distribution along any direction in the quarter plane, see \cite[Theorems 22--28]{franceschi_asymptotic_2016}. Constants in these asymptotics involve the functions $\phi_1$ and $\phi_2$ in \eqref{eq:Laplace_transform_boundary}, and can thus be made explicit with Theorem~\ref{thm:main}.
     \item It is also interesting to dissect \eqref{eq:main_formula_with_constants} and to notice that certain quantities in that formula depend only on the behavior of the process in the interior of the quadrant ($w$ and $\R^-$), while the remaining ones mix properties of the interior and boundary of the quarter plane ($\nu_1(\mathbb{R}_+)$, $p$, $\chi$ and $G$).
     \item The statement of Theorem~\ref{thm:main} (namely, an expression of the Laplace transform as a Cauchy integral), as well as the techniques we shall employ to prove it (viz., reduction to BVPs with shift), are reminiscent of the results and methods used for discrete random walks in the quarter plane, see \cite{fayolle_random_1999} for a modern reference, and \cite{malysev_analytic_1972,fayolle_two_1979,kurkova_malyshevs_2003} for historical breakthroughs.
     \item Altogether, Theorem~\ref{thm:main} illustrates that the analytic approach consisting in solving quarter plane problems via BVPs is better suited for diffusions than for discrete random walks. We can actually treat any wedge, covariance matrix, drift vector and reflection vectors (see Corollary \ref{cor:main} below), whereas in the discrete case, hypotheses should be done on the boundedness of the jumps (only small steps are considered in \cite{fayolle_random_1999,kurkova_malyshevs_2003,BMMi-10,bernardi_counting_2015}) and on the cone (typically, half and quarter planes only).
\end{itemize}

Theorem \ref{thm:main} further leads to an explicit expression for the Laplace transform of the stationary distribution of reflected Brownian motion in an arbitrary convex wedge, as stated in the following corollary, whose proof is postponed to Appendix \ref{app:BM_cones}.
\begin{cor}
\label{cor:main}
Let $\widetilde{Z}$ be a reflected Brownian motion in a wedge of angle $\beta\in(0,\pi)$, of covariance matrix $\widetilde{\Sigma}$, drift $\widetilde{\mu}$ and reflection matrix $\widetilde{R}=(\widetilde{R}^1,\widetilde{R}^2)$, corresponding to the angles of reflection $\delta$ and $\epsilon$ on Figure \ref{fig:linear_transformation}. Assume it is recurrent and note $\widetilde{\pi}$ the stationary distribution and $\widetilde{\phi}$ its Laplace transform. Let
\begin{equation*}
%\label{eq:linear_transforms_T}
     T_1= 
\begin{pmatrix}
   \frac{1}{\sin \beta}& \cot \beta \\
   0 & 1
\end{pmatrix}
%\begin{pmatrix}
%   \frac{1}{\sqrt{\sigma_{11}}} & 0 \\
%   0 &\frac{1}{\sqrt{\sigma_{22}}}
%\end{pmatrix}
,\qquad
T_1^{-1}= 
%\begin{pmatrix}
%  \sqrt{\sigma_{11}} & 0 \\
%   0 &\sqrt{\sigma_{22}}
%\end{pmatrix}
\begin{pmatrix}
   \sin \beta & -{\cos \beta} \\
   0 & \phantom{-}1
\end{pmatrix}
.
\end{equation*}
Then
\begin{equation*}
     \widetilde{\phi}(\widetilde{\theta})=\phi(T_1^\top \widetilde{\theta}),
\end{equation*}
where $\phi$ is the Laplace transform \eqref{eq:main_formula_with_constants} of Theorem \ref{thm:main} associated to $(\Sigma,\mu,R)$, with 
\begin{equation*}
\Sigma=T_1\widetilde{\Sigma}T_1^\top
%=\begin{pmatrix}
%   1 & -{\cos \beta} \\
%   -{\cos \beta} & 1
%\end{pmatrix}
, \quad
\mu=T_1^{-1}\widetilde{\mu}
%=
%\begin{pmatrix}
%   \mu_1\sin \beta- \mu_2 \cos\beta \\
%   \mu_2
%\end{pmatrix}
\quad \text{and} \quad
R=T_1^{-1}\widetilde{R}.
\end{equation*}
\end{cor}

\subsection*{Structure of the paper}
\begin{itemize}
     \item Section \ref{sec:kernel}: analytic preliminaries, continuation of the Laplace transforms and definition of an important hyperbola
     \item Section \ref{sec:Riemann}: statement and proof that the Laplace transforms satisfy BVP of Carleman-type on branches of hyperbolas, transformation of the Carleman BVP with shift into a (more classical) Riemann BVP, study of the conformal mapping allowing this transformation, resolution of the BVP
     \item Section \ref{sec:asymptotics}: asymptotics of the stationary distribution, links with Dai and Miyazawa's \cite{DaMi-13} asymptotic results
     \item Section \ref{sec:algebraic_nature}: simplifications of the integral expression of Theorem \ref{thm:main} for models satisfying the skew-symmetric condition, for orthogonal reflections (leading to a new proof of the results of \cite{franceschi_tuttes_2016}), links with Dieker and Moriarty's results \cite{DiMo-09}
     \item Appendix \ref{app:BM_cones}: equivalence between Brownian motion in the quarter plane and Brownian motion in convex wedges
\end{itemize}

\subsection*{Acknowledgements}
We thank Irina Kurkova and Yuri Suhov for interesting discussions. We acknowledge support from the ``projet MADACA'' (2014--2016), funded by the R\'egion Centre-Val de Loire (France). \textcolor{black}{Finally we thank the editors and referees for their useful remarks and suggestions.}

\section{Methodology and analytic preliminaries}
\label{sec:kernel}

\subsection{Methodology and positioning of our work regarding \cite{Foschini,foddy_1984,Baccelli_Fayolle_1987,dai_reflecting_2011}}
\label{sec:Methodology}

Schematically, our argument for the proof of Theorem \ref{thm:main} is composed of the following steps: 
\begin{enumerate}[label={\rm (\arabic{*})},ref={\rm (\arabic{*})}]
     \item\label{it:pre}presentation of the functional equation, analytic preliminaries, meromorphic continuation of the Laplace transforms (Section \ref{sec:kernel});
     \item\label{it:Carleman_BVP}statement of a Carleman BVP with shift satisfied by the Laplace transforms (Section~\ref{sec:Carleman}); 
     \item\label{it:conformal_mapping}introduction of a conformal mapping, allowing to transform the latter BVP into a more classical Riemann BVP (Section \ref{sec:gluing}), see Figure \ref{fig:different_levels}; 
     \item\label{it:Riemann_BVP}statement of the Riemann BVP (Section \ref{sec:Riemann_BVP}); 
     \item\label{it:index}definition and study of the index (denoted by $\chi$ in Theorem \ref{thm:main}), which turns out to have a crucial role in solving the Riemann BVP (Section \ref{sec:index}); 
     \item\label{it:resolution}resolution of the BVP (Section \ref{sec:resolution}).
\end{enumerate}
Except for the study of the index (item \ref{it:index}, which is specific to our problem at hand), the above structuration of the proof dates back to (in chronological order) \cite{fayolle_two_1979} (for the discrete setting, which later led to the book \cite{fayolle_random_1999}), \cite{Foschini}, \cite{foddy_1984} and \cite{Baccelli_Fayolle_1987}.

To begin the discussion, let us remind that the process in \cite{Baccelli_Fayolle_1987} is absorbed on the boundary and then relaunched in the quadrant after an exponential time; it is therefore totally different from ours. Let us also recall that the analytic study of reflected Brownian motion in the quadrant was initiated in \cite{Foschini,foddy_1984}, but both \cite{Foschini} and \cite{foddy_1984} did quite restrictive assumptions on the process (symmetry conditions $\mu_1 = \mu_2$, $\sigma_{11}=\sigma_{22}$, and symmetric reflection vectors in \cite{Foschini}, identity covariance matrix in \cite{foddy_1984}).

In this paper we follow the same steps \ref{it:pre}--\ref{it:resolution} and use a synthetic approach of \cite{fayolle_two_1979,Foschini,foddy_1984,Baccelli_Fayolle_1987}, to eventually go further than the existing literature and prove Theorem \ref{thm:main}. Some technical details of the present work are borrowed from \cite{Foschini,foddy_1984,Baccelli_Fayolle_1987} and more particularly from \cite{Baccelli_Fayolle_1987}. Although being different, the stochastic process of \cite{Baccelli_Fayolle_1987} and ours share the property of satisfying a functional equation with the same left-hand side (as in \eqref{eq:functional_equation}). Accordingly, technical details will be very similar as soon as they concern the kernel; this is the case of Section~\ref{sec:kernel}. \textcolor{black}{In particular, as in \cite{Baccelli_Fayolle_1987}, the Laplace transforms will satisfy BVP on hyperbolas.}

On the other hand, many points of our analysis profoundly differ from that of \cite{Baccelli_Fayolle_1987}: the resolution of the BVP is different as the right-hand side of \eqref{eq:functional_equation} is not comparable to that of \cite{Baccelli_Fayolle_1987}. Moreover we make our main result as tractable as possible: see our Section \ref{sec:algebraic_nature}, where we present many techniques to simplify the analytic expression derived in Theorem \ref{thm:main}. We further also propose asymptotic considerations, in close relation to those of \cite{dai_reflecting_2011}.

\textcolor{black}{We shall prove in full detail our main results only in the case where both coordinates of the drift are negative:
\begin{equation}
\label{eq:drift_negative}
     \mu_1<0,\qquad \mu_2<0.
\end{equation}
This hypothesis (also done in \cite{Foschini,foddy_1984,franceschi_asymptotic_2016}) is only technical, and allows us to reduce the number of cases to handle. In Section \ref{subsec:generalizations} we comment on the case of a drift $\mu$ with one non-negative coordinate (having two non-negative coordinates is incompatible with \eqref{eq:stationary_distribution_CNS}), and we explain how Theorem \ref{thm:main} remains valid in this case.}

\subsection{\textcolor{black}{Real points of the kernel}}
\label{subsec:functional_equation}

The set of real points of the zero set of the kernel
\begin{equation*}
     \{(\theta_1,\theta_2)\in\mathbb R^2: \gamma(\theta_1,\theta_2)=0\}
\end{equation*}
defines an ellipse, see Figure \ref{fig:p_ellipse}. Introduced in \cite{dai_reflecting_2011}, this ellipse offers the possibility of presenting many analytical quantities in a clear and compact way:
\begin{itemize}
     \item the branch points $\theta_1^-$ and $\theta_1^+$ (resp.\ $\theta_2^-$ and $\theta_2^+$) in \eqref{eq:definition_theta_pm} are the leftmost and rightmost (resp.\ bottommost and topmost) points on the ellipse;
     \item the drift $\mu$ is orthogonal to the tangent at the origin;
     \item the set of points where $\gamma_1(\theta_1,\theta_2)=0$ and $\gamma_2(\theta_1,\theta_2)=0$ are straight lines, orthogonal to the reflection vectors $R^1$ and $R^2$. Their intersection points with the ellipse are easily computed (notice that these intersection points appear in the statement of Theorem \ref{thm:main}, in particular in the index $\chi$).
\end{itemize}

\begin{figure}[h!] 
\centering
\includegraphics[scale=1]{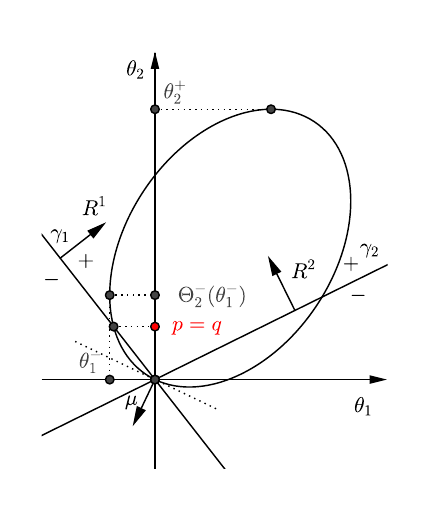}
\includegraphics[scale=1]{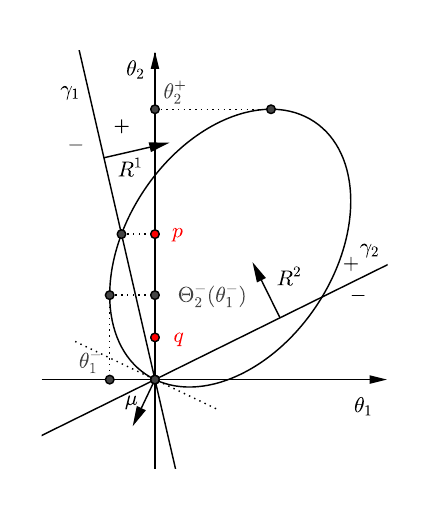}
\caption{The ellipse and the two straight lines are the sets of real points $(\theta_1,\theta_2)$ which cancel $\gamma$, $\gamma_1$ and $\gamma_2$, respectively. Location of $p$ (see \eqref{def:definition_p}) and $q$ (see \eqref{eq:def_q}) according to the sign of $\gamma_1(\theta_1^-,\Theta_2^-(\theta_1^-))$}
\label{fig:p_ellipse}
\end{figure}

\subsection{Meromorphic continuation and poles of the Laplace transforms}
\label{subsec:continuation}
In Section~\ref{sec:Carleman} we shall state a boundary condition for $\phi_1$, on a curve lying outside its natural domain of definition (namely, the half plane with negative real part). The statement hereafter (straightforward consequence of the functional equation \eqref{eq:functional_equation}, see Proposition \ref{prop:continuation_Laplace_transform} for an extended version) proposes a meromorphic continuation on a domain containing the latter curve.

\begin{lem}
\label{lem:continuation}
The Laplace transform $\phi_1(\theta_2)$ can be extended meromorphically to the open and simply connected set 
\begin{equation}
\label{eq:domain_continuation}
     \{\theta_2\in\textcolor{black}{\mathbb{C} \setminus (\theta_2^+,\infty)} : \Re\,\theta_2\leqslant 0 \text{ or } \Re\, \Theta_1^-(\theta_2) <0\},
\end{equation}
by mean of the formula
\begin{equation}
\label{eq:continuation_formula}
     \phi_1(\theta_2)=-\frac{\gamma_2}{\gamma_1}(\Theta_1^-(\theta_2),\theta_2)\phi_2(\Theta_1^-(\theta_2)).
\end{equation}
\end{lem}

\begin{proof}
\textcolor{black}{The domain \eqref{eq:domain_continuation} is simply connected by \cite[\S 2.4]{franceschi_asymptotic_2016}. The formula \eqref{eq:continuation_formula}, see \cite[Lemma 6]{franceschi_asymptotic_2016}, is a direct consequence of the functional equation \eqref{eq:functional_equation} evaluated at $(\Theta_1^-(\theta_2),\theta_2)$, first on the (non-empty) open domain
\begin{equation*}
     \{\theta_2\in\mathbb{C} : \Re\,\theta_2< 0 \text{ and } \Re\, \Theta_1^-(\theta_2) <0\}.\qedhere
\end{equation*}}
\end{proof}
\textcolor{black}{Due to the continuation formula \eqref{eq:continuation_formula}, the only possible pole of $\phi_1$ in the domain \eqref{eq:domain_continuation} will come from a cancelation of the denominator $\gamma_1$.} \textcolor{black}{More precisely,} let $p$ be the (unique, when it exists) non-zero point such that %(cf.\ Figure \ref{fig:p_ellipse}) 
\begin{equation} 
\label{def:definition_p}
     \gamma_1(\Theta_1^- (p),p)=0.
\end{equation}
\textcolor{black}{
It follows from \eqref{def:definition_p} that $p$ satisfies a second degree polynomial equation with real coefficients. As one of the roots is $0$ the other one must be $p$, which is then real and equals (when it exists)
\begin{equation}
\label{eq:expression_p}
     p=\frac{2r_{11}(\mu_1r_{21}-\mu_2r_{11})}{r_{11}^2\sigma_{22}-2r_{11}r_{21}\sigma_{12}+r_{21}^2\sigma_{11}}.
\end{equation}}\textcolor{black}{Formula \eqref{def:definition_p} means that $p$ is the ordinate of the intersection point between the ellipse $\gamma=0$ and the line $\gamma_1=0$, see Figure \ref{fig:p_ellipse}. The intersection point always exists but sometimes its abscissa is associated to $\Theta_1^-$ (when $p$ exists, i.e., $\gamma_1(\Theta_1^\pm(\theta_2^+),\theta_2^+)\geqslant 0$) and sometimes to $\Theta_1^+$ (when $p$ doesn't exist, i.e., $\gamma_1(\Theta_1^\pm(\theta_2^+),\theta_2^+)< 0$), the limit case being $p=\theta_2^+$, see Figure \ref{fig:p_ellipse_bis}.}
\begin{figure}[hbtp] 
\centering
\includegraphics[scale=1]{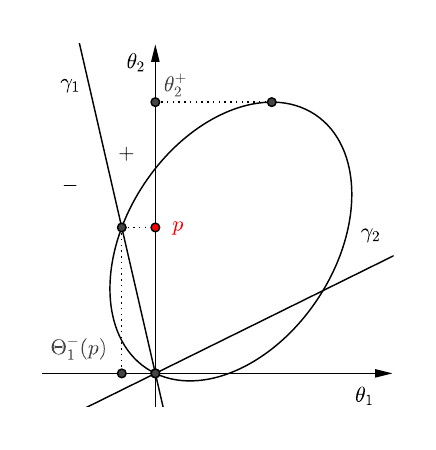}
\includegraphics[scale=1]{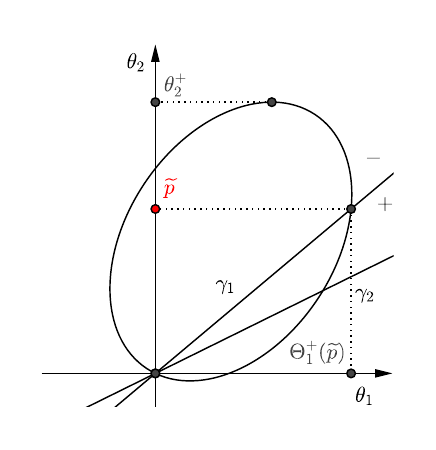}
\caption{The intersection point between the ellipse $\gamma=0$ and the  straight line $\gamma_1=0$. On the left side $p$ exists. \textcolor{black}{On the right side $p$ doesn't exist, although $\widetilde p$,} \textcolor{black}{defined by that $\gamma_1(\Theta_1^+(\widetilde p), \widetilde p)=0$, exists}}
\label{fig:p_ellipse_bis}
\end{figure}

\textcolor{black}{Let us finally remark that the pole that $\phi_1$ may have at $p$ is necessarily simple, due to the expression \eqref{eq:def_gamma_gamma1_gamma2} of $\gamma_1$.}

\subsection{An important hyperbola}

For further use, we need to introduce the curve
\begin{equation}
\label{eq:curve_definition}
     \R =\{\theta_2\in\mathbb C: \gamma(\theta_1,\theta_2)=0 \text{ and } \theta_1\in(-\infty,\theta_1^-)\}=\Theta_2^\pm ((-\infty,\theta_1^-)).
\end{equation}
It is symmetrical w.r.t.\ the horizontal axis, see Figure~\ref{fig:BVP_theta}. Indeed, the discriminant of $\Theta_2^\pm$ (i.e., the polynomial under the square root in \eqref{eq:definition_Theta_pm}) is positive on $(\theta_1^-, \theta_1^+)$ and negative on ${\mathbb R} \setminus [\theta_1^-, \theta_1^+]$. Accordingly, the branches $\Theta_2^\pm$ take respectively real and complex conjugate values on the sets above. Furthermore, $\R$ has a simple structure, as shown by the following elementary result:
\begin{lem}[Lemma 9 in \cite{Baccelli_Fayolle_1987}]
\label{lem:Lemma9BaFa87}
The curve $\R$ in~\eqref{eq:curve_definition} is a branch of hyperbola, whose equation is
\begin{equation}
\label{eq:hyperbole}
     \sigma_{22}(\sigma_{12}^2-\sigma_{11}\sigma_{22})x^2+\sigma_{12}^2\sigma_{22}y^2-2\sigma_{22}(\sigma_{11}\mu_2-\sigma_{12}\mu_1)x=\mu_2(\sigma_{11}\mu_2-2\sigma_{12}\mu_1).
\end{equation}
\textcolor{black}{The curve $\R$ is the right branch of the hyperbola if the covariance factor $\sigma_{12}>0$, the left branch if $\sigma_{12}<0$, and a vertical straight line in the limit case $\sigma_{12}=0$.}
\end{lem}
\begin{figure}[hbtp]
    \centering
    \begin{subfigure}{0.3\textwidth}
        \includegraphics[scale=0.8]{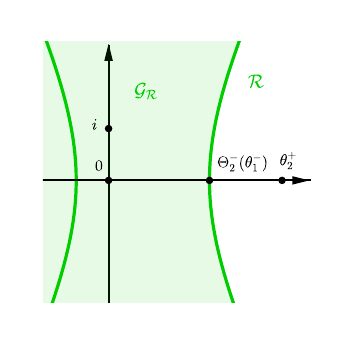}
        \caption{$\sigma_{12}>0$}
        \label{fig:gull2}
    \end{subfigure}%\quad
    \begin{subfigure}{0.3\textwidth}
        \includegraphics[scale=0.8]{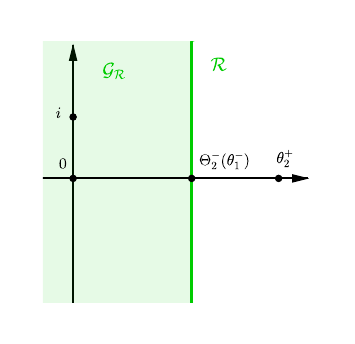}
        \caption{$\sigma_{12}=0$}
        \label{fig:tiger2}
    \end{subfigure}\quad
    \begin{subfigure}{0.3\textwidth}
\includegraphics[scale=0.8]{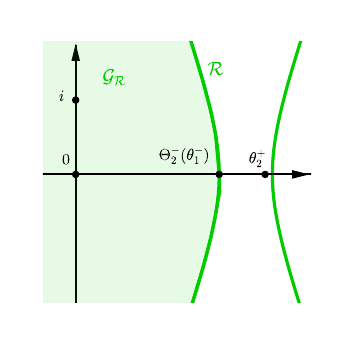}
        \caption{$\sigma_{12}<0$}
        \label{fig:mouse2}
    \end{subfigure}
\caption{\textcolor{black}{The shape (in particular the orientation) of the curve $\R$ depends directly on the sign of the covariance $\sigma_{12}$}}
\end{figure}

We denote the part of $\R$ with negative imaginary part by $\R^-$, see \eqref{eq:curve_definition_negative} and Figure \ref{fig:BVP_theta}. 
We further denote by $\G$ the open domain of $\mathbb{C}$ containing $0$ and bounded by $\R$, see again Figure~\ref{fig:BVP_theta}. The closure of $\G$ is equal to $\G\cup\R$ and will be noted $\overline{\G}$. 

\begin{lem}
\label{lem:cont}
\textcolor{black}{The domain 
\begin{equation*}
     \{\theta_2\in\mathbb{C} \setminus (\theta_2^+,\infty) : \Re\,\theta_2\leqslant 0 \text{ or } \Re\, \Theta_1^-(\theta_2) <0\},
\end{equation*}
defined in \eqref{eq:domain_continuation}, strictly contains $\overline{\G}$.}
%, see \cite{franceschi_asymptotic_2016}.
\end{lem}

%Lemmas \ref{lem:continuation} and \ref{lem:cont} together 
Thanks to Lemma \ref{lem:continuation}, this implies that the Laplace transform $\phi_1(\theta_2)$ can be extended meromorphically to a domain containing $\overline{\G}$.
%\begin{proof}[Proof of Lemma \ref{lem:cont}]

\begin{proof}
The first way to see the above inclusion is to refer to \cite{franceschi_asymptotic_2016}, where the boundaries of the domain \eqref{eq:domain_continuation} (called $\Delta \cup \{s_0\}$ in \cite{franceschi_asymptotic_2016}) are computed, see in particular \cite[Figures 10 and 11]{franceschi_asymptotic_2016}. The technique used in \cite{franceschi_asymptotic_2016} is to use the parametrization of the zero set of the kernel \eqref{eq:def_gamma_gamma1_gamma2}.

\textcolor{black}{It is also possible to show it directly and more elementarily, as follows.
First of all, the set $\overline{\G} \cap \{\theta_2\in\mathbb{C} : \Re\, \theta_2 \leqslant 0\}$ is obviously included in the domain defined in \eqref{eq:domain_continuation}. It thus remains to show that the set $\overline{\G} \cap \{\theta_2\in\mathbb{C} : \Re\, \theta_2 > 0\}$, which is bounded by (a part of) $\R$ and (a part of) $i\mathbb{R}$, see Figure~\ref{fig:domainprolongement}, is included in the domain \eqref{eq:domain_continuation}. More specifically, we are going to prove that the latter set is a subset of
\begin{equation*}
     \{\theta_2\in\mathbb{C} \setminus (\theta_2^+,\infty) : \Re\, \Theta_1^-(\theta_2) <0\}.
\end{equation*}
First, the definition \eqref{eq:curve_definition} of $\R$ obviously implies that $\R \subset \{\theta_2\in\mathbb{C} \setminus (\theta_2^+,\infty) : \Re\, \Theta_1^-(\theta_2) <0\}$. In the same way, we notice that $i\mathbb{R}$ also belongs to that set. Indeed, for $t\in\mathbb{R}$, Equation \eqref{eq:definition_Theta_pm} yields
\begin{equation*}
     \Re \Theta_1^-(it)=\frac{1}{\sigma_{11}} \left( -\mu_1 - \Re
\sqrt{\mu_1^2+ t^2 \det \Sigma 
+2it(\mu_1\sigma_{12}-\mu_2\sigma_{11})} \right) <0,
\end{equation*}
because $\mu_1<0$ by \eqref{eq:drift_negative}, and since $\Sigma$ is a covariance matrix, $\det \Sigma >0$ and $\sigma_{11}>0$.
Then there are two cases to consider:}
\begin{itemize}
     \item\textcolor{black}{$\sigma_{12}<0$ (i.e., $0<\beta<\frac{\pi}{2}$): the set $\overline{\G} \cap \{\theta_2\in\mathbb{C} : \Re\, \theta_2 > 0\}$ is bounded, see the left picture on Figure~\ref{fig:domainprolongement}. Then the maximum principle applied to the function $\Re \Theta_1^-$ %(which is analytic on $\mathbb{C}\setminus ((-\infty,\theta_2^-]\cup[\theta_2^+,\infty))$) 
     implies that the image of every point of the set $\overline{\G} \cap \{\theta_2\in\mathbb{C} : \Re\, \theta_2 > 0\}$ by $\Re \Theta_1^-$ is negative.}
     \item\textcolor{black}{$\sigma_{12}\geqslant 0$ (i.e., $\frac{\pi}{2}\leqslant \beta<\pi$): it is no more possible to apply directly the maximum principle, as the set $\overline{\G} \cap \{\theta_2\in\mathbb{C} : \Re\, \theta_2 > 0\}$ is now unbounded, see the right display on Figure \ref{fig:domainprolongement}. However, to conclude it is enough to show that the image by $\Re \Theta_1^-$ of a point $re^{it}$ near to infinity and in $\overline{\G} \cap \{\theta_2\in\mathbb{C} : \Re\, \theta_2 > 0\}$ is negative.}
     
    \textcolor{black}{The asymptotic directions of $i\mathbb{R}$ (resp.\ $\R$) are $\pm \frac{\pi}{2}$ (resp.\ $\pm(\pi-\beta)$), as this comes from \eqref{eq:hyperbole} and \eqref{eq:definition_beta}. Then we prove that $\Re\Theta_1^-(re^{\pm it})<0$, for $r$ large enough and $t\in(\pi-\beta, \frac{\pi}{2})$.
For $ t\in(0,\pi)$, the formula \eqref{eq:definition_Theta_pm} gives the following limit:
\begin{equation*}
\underset{r\to\infty}{\lim}
\frac{\Theta_1^-(re^{\pm it})}{re^{\pm it}}
=
\frac{-\sigma_{12}\pm i\sqrt{\det \Sigma}}{\sigma_{11}}=\sqrt{\frac{\sigma_{22}}{\sigma_{11}}} e^{\pm i\beta}.
\end{equation*}
%where the sign $\pm$ comes from the cut of the definition domain of $\Theta_1^-$ due to the square root.
Taking $ t\in  (\pi-\beta, \frac{\pi}{2})$ we see that
\begin{equation*}
\Theta_1^-(re^{\pm it})\underset{r\to\infty}{\sim}
r\sqrt{\frac{\sigma_{22}}{\sigma_{11}}} e^{\pm i(t+\beta)},
\end{equation*}
and since $t+\beta \in (\pi, \frac{\pi}{2}+\beta) \subset (\pi, \frac{3\pi}{2})$ we obtain that $\Re \Theta_1^-(re^{\pm it})<0$ for $r$ large enough. We conclude the proof with the maximum principle as in the case $\sigma_{12}<0$.\qedhere}
\end{itemize}\end{proof}

\begin{figure}[hbtp]
\centering
\includegraphics[scale=0.2]{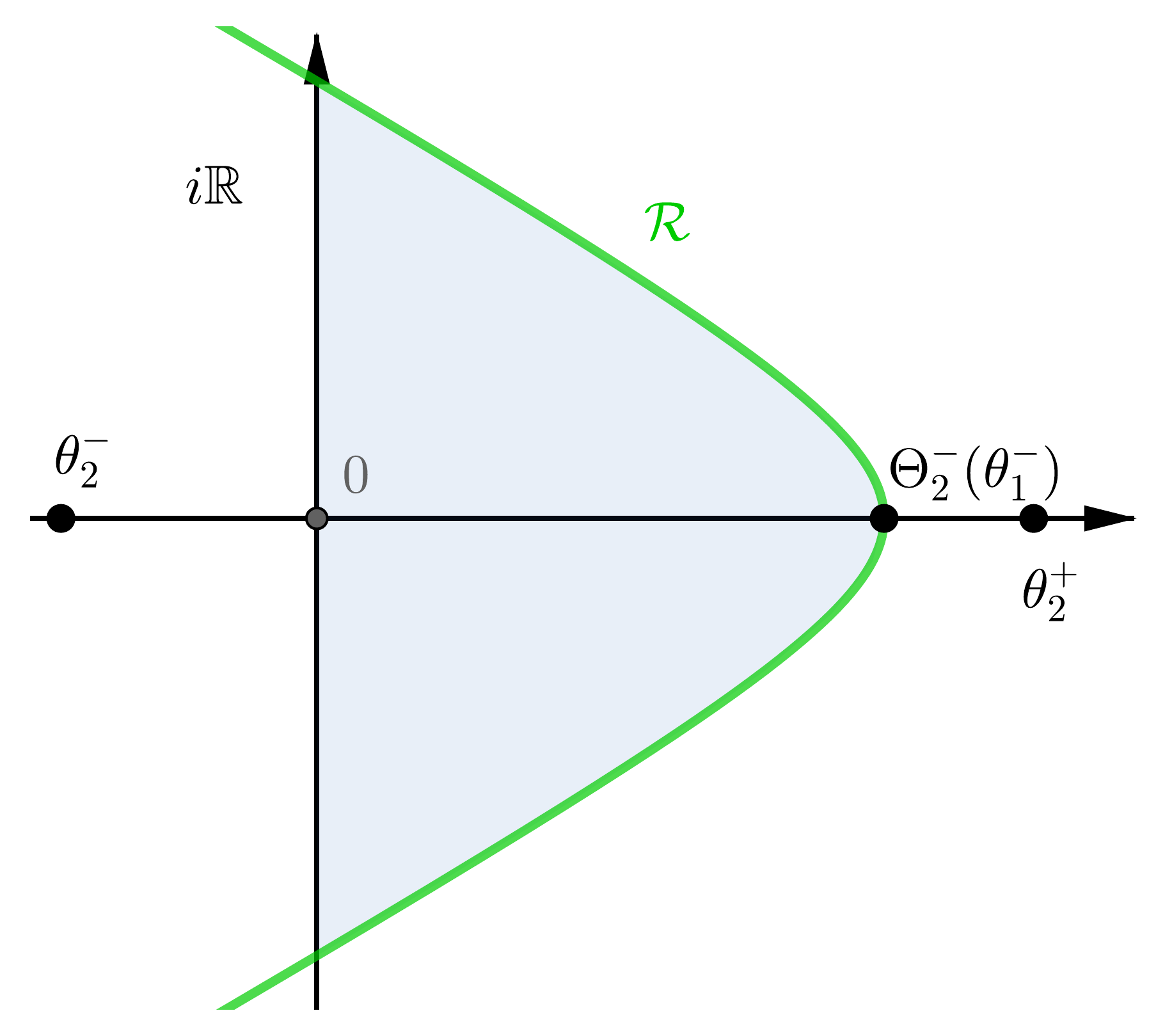}\qquad
\includegraphics[scale=0.2]{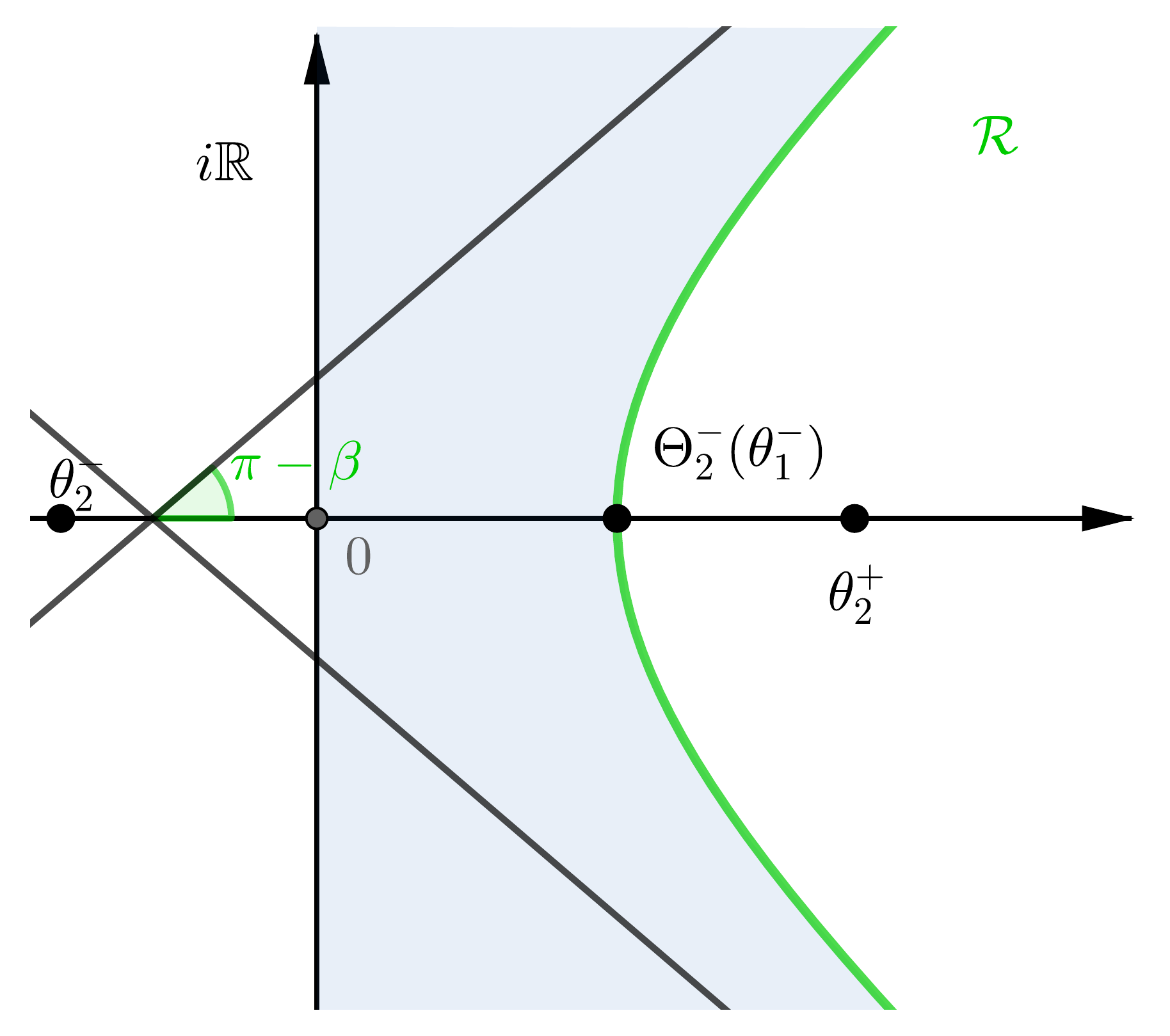}
\caption{\textcolor{black}{On the left $\sigma_{12}<0$, and on the right $\sigma_{12}\geqslant 0$. The blue domain is the set $\overline{\G} \cap \{\theta_2\in\mathbb{C} : \Re\, \theta_2 >0\}$. On the right, the two lines are the asymptotes to $\R$}}
\label{fig:domainprolongement}
\end{figure}

\begin{figure}[ht!] 
\centering
\includegraphics[scale=1]{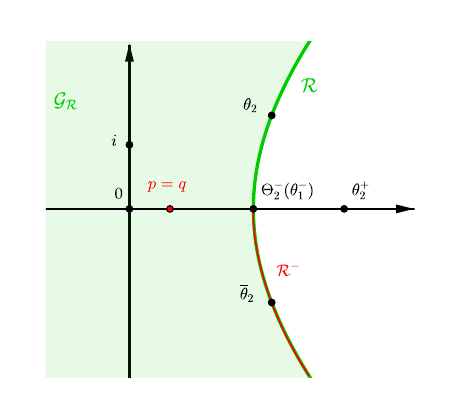}
\includegraphics[scale=1]{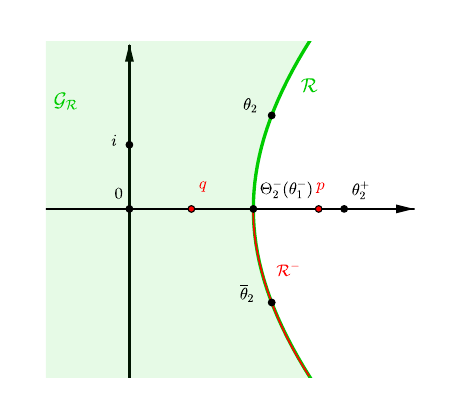}
\caption{The curve $\R$ in~\eqref{eq:curve_definition} is symmetric w.r.t.\ the horizontal axis, and $\G$ is the domain in green. The curve $\R^-$ is the half branch of $\R$ with negative imaginary part. The points $p$ and $q$ are used to define the conformal mapping $W$, see \eqref{def:definition_p} and \eqref{eq:definition_W}. \textcolor{black}{The point $q$ is defined as follows:} if the pole $p$ is in $\G$ then $q=p$ (figure on the left), otherwise $q=\frac{1}{2}\Theta_2^-(\theta_1^-)$ (figure on the right), see \eqref{eq:def_q}}
\label{fig:BVP_theta}
\end{figure}

\section{A proof of Theorem \ref{thm:main} via reduction to BVPs}
\label{sec:Riemann}

\subsection{Carleman BVP}
\label{sec:Carleman}

For $\theta_2\in \R$, define the function $G$ as in \eqref{eq:definition_G}:
\begin{equation*}
     G(\theta_2)=\frac{\gamma_1}{\gamma_2}(\Theta_1^-(\theta_2),\theta_2)\frac{\gamma_2}{\gamma_1}(\Theta_1^-(\overline{\theta_2}),\overline{\theta_2}).
\end{equation*}
Notice that $G(\theta_2)G(\overline{\theta_2})=1$ and that for $\theta_2\in \R$ one has $\Theta_1^-(\overline{\theta_2})=\Theta_1^-({\theta_2})$. Let us also recall that $p$ is defined in \eqref{def:definition_p}.

\begin{prop}[Carleman BVP with shift]
\label{prop:BVP_Carleman}
The function $\phi_1$ in \eqref{eq:Laplace_transform_boundary}
\begin{enumerate}[label={\rm (\arabic{*})},ref={\rm (\arabic{*})}]
     \item\label{item:BVPCarleman_1}
     is meromorphic on $\mathcal{G}_{\mathcal{R}}$, 
\begin{itemize}
\item without pole on $\G$ if $\gamma_1(\theta_1^-,\Theta_2^-(\theta_1^-))<0$,
\item with a single pole on $\G$ at $p$ of order one if $\gamma_1(\theta_1^-,\Theta_2^-(\theta_1^-))>0$,
\item  without pole on $\G$ and with a single pole of order one on the boundary $\R$ of $\G$, at $p=\Theta_2^-(\theta_1^-)$, if $\gamma_1(\theta_1^-,\Theta_2^-(\theta_1^-))=0$,
\end{itemize}     
     \item\label{item:BVPCarleman_2} is continuous on $\overline{\mathcal{G}_{\mathcal{R}}}\setminus \{p\}$ and bounded at infinity,
     \item\label{item:BVPCarleman_3} satisfies the boundary condition
\begin{equation} 
\label{eq:bound_cond_gen}
     \phi_1(\overline{\theta_2})=G(\theta_2)\phi_1({\theta_2}), \qquad\forall \theta_2\in \mathcal{R}.
\end{equation}
\end{enumerate}
\end{prop}

It is worth mentioning that the condition on the sign of $\gamma_1(\theta_1^-,\Theta_2^-(\theta_1^-))$ has a clear geometric meaning: indeed, $\gamma_1(\theta_1^-,\Theta_2^-(\theta_1^-))$ is negative (resp.\ positive) if and only if the straight line corresponding to $\gamma_1$ crosses the ellipse below (resp.\ above) the ordinate $\Theta_2^-(\theta_1^-)$; see Figure \ref{fig:p_ellipse}, left (resp.\ right).

\begin{rem}
\label{rem:different_cases_singularities}
Item \ref{item:BVPCarleman_1} of Proposition \ref{prop:BVP_Carleman} shows that according to the values of the parameters, various cases exist regarding the singularities of the Laplace transform in the domain $\G$. This is the reason why there isn't a unique expression for the Laplace transform in our main Theorem~\ref{thm:main}, but two different expressions.
\end{rem}

\begin{proof}[Proof of Proposition \ref{prop:BVP_Carleman}]
First of all, it follows from Lemma \ref{lem:continuation} that $\phi_1$ is meromorphic in $\G$ and may have a pole of order one at $p$. Indeed, due to the continuation formula \eqref{eq:continuation_formula}, the only potential pole $p$ of $\phi_1$ in $\G$ should be a zero of $\gamma_1$. It is then on the real line and characterized by \eqref{def:definition_p}. Moreover, $p$ defined by \eqref{def:definition_p} is smaller than $\Theta_2^-(\theta_1^-)$ (i.e., $p\in\G$) if and only if the geometric condition $\gamma_1(\theta_1^-,\Theta_2^-(\theta_1^-))>0$ is satisfied, see Figure~\ref{fig:p_ellipse}. This demonstrates the first item of Proposition \ref{prop:BVP_Carleman}.

The second item (in particular the fact that $\phi_1$ is bounded at infinity)\ comes from Lemma~\ref{lem:continuation} together with the fact that \eqref{eq:Laplace_transform_boundary} implies that $\phi_1$ (resp.\ $\phi_2$) is bounded on the set $\{\theta_2 \in \mathbb{C} : \Re\,\theta_2 \leqslant 0\}$ (resp.\ $\{\theta_2 \in \mathbb{C} : \Re\,\theta_1 \leqslant 0\}$). 

To prove the boundary condition \eqref{eq:bound_cond_gen} (that we announced in \cite[Proposition 7]{franceschi_tuttes_2016}), we consider $\theta_1$ such that ${\Re}\,\theta_1 <0$, and evaluate the functional equation \eqref{eq:functional_equation} at $(\theta_1,\Theta_2^\pm(\theta_1))$. This implies
\begin{equation*}
     \frac{\gamma_1}{\gamma_2}(\theta_1,\Theta_2^\pm(\theta_1))\phi_1(\Theta_2^\pm(\theta_1))+\phi_2(\theta_1)=0,
\end{equation*}
which in turn yields
\begin{equation}
\label{eq:equality_+_-}
     \frac{\gamma_1}{\gamma_2}(\theta_1,\Theta_2^+(\theta_1))\phi_1(\Theta_2^+(\theta_1))
     =
     \frac{\gamma_1}{\gamma_2}(\theta_1,\Theta_2^-(\theta_1))\phi_1(\Theta_2^-(\theta_1)).
\end{equation}
Restricting \eqref{eq:equality_+_-} to values of $\theta_1\in (-\infty,\theta_1^-)$, for which $\Theta_2^+ (\theta_1)$ and $\Theta_2^-(\theta_1)$ are complex conjugate (see Section \ref{sec:kernel}), \textcolor{black}{noting $\theta_2=\Theta_2^-(\theta_1) \in\R$ and noticing that $\theta_1=\Theta_1^-(\theta_2)$,} we reach the conclusion that
\begin{equation*}
     \phi_1(\overline{\theta_2})
     =
     \frac{\gamma_1}{\gamma_2}(\theta_1,\theta_2)\frac{\gamma_2}{\gamma_1}(\theta_1,\overline{\theta_2})\phi_1({\theta_2}),
\end{equation*}
which, by definition \eqref{eq:definition_G} of $G$, exactly coincides with the boundary condition \eqref{eq:bound_cond_gen}. \textcolor{black}{Although we do not exclude a priori the denominators in \eqref{eq:equality_+_-} to vanish, note that this does not happen for $\theta_1\in(-\infty,\theta_1)$ since then the imaginary part of $\Theta_2^\pm(\theta_1)$ is non-zero.}
\end{proof}

The BVP established in Proposition \ref{prop:BVP_Carleman} belongs to the class of homogeneous Carleman (or Riemann-Carleman) BVPs with shift, see \cite{litvinchuk_solvability_2000}, the shift being here the complex conjugation. 

In some cases, the function $G$ in \eqref{eq:bound_cond_gen} can be factorized, leading to an interesting particular case, that we comment below. \textcolor{black}{As we shall see, the well-known skew-symmetric condition
\begin{equation}
\label{eq:skew-symmetric}
     2\Sigma = R\cdot \text{diag}(R)^{-1}\cdot \text{diag} (\Sigma)+\text{diag} (\Sigma)\cdot\text{diag}(R)^{-1}\cdot R^{\top}
\end{equation}
 (equivalent for the stationary distribution $\pi(x_1,x_2)$ to have a product-form) gives a family of examples where such a factorization holds.
In \eqref{eq:skew-symmetric} we have noted $\text{diag} (A)$ the diagonal matrix with the same diagonal coefficients as those of $A$.}

\begin{rem}
\label{rem:decoupling}
If there exists a rational function $F$ such that 
\begin{equation*}
     G(\theta_2)=\frac{F(\theta_2)}{F(\overline{\theta_2})},
\end{equation*}
one can transform the boundary condition~\eqref{eq:bound_cond_gen} for $\phi_1$ with $G\neq1$ into a boundary condition for $\phi_1 \cdot F$ with $G=1$, namely,
\begin{equation*}
     (\phi_1 \cdot F)(\overline{\theta_2})=(\phi_1 \cdot F)({\theta_2}).
\end{equation*}
Then the associated BVP \textcolor{black}{should be solvable} using Tutte's invariants \cite{bernardi_counting_2015}. This is what has been done in \cite{franceschi_tuttes_2016}, for the particular case of orthogonal reflections (corresponding to $F(\theta_2)=\frac{1}{\theta_2}$). 

\textcolor{black}{We show in Section \ref{subsec:skew} that such a rational function $F$ always exists in the skew-symmetric case \eqref{eq:skew-symmetric}, and from this we derive a rational expression of the Laplace transform.}

The existence of a rational function $F$ factorizing $G$ as above is reminiscent of the notion of decoupling function or telescoper, introduced in \cite{bernardi_counting_2015,Dreyfus}.
\end{rem}

However, a rational factorization term $F$ as in Remark \ref{rem:decoupling} does not exist in general, and it is still an open problem to characterize the parameters $(\Sigma,\mu,R)$ for which $F$ exists. As a consequence, we cannot systematically use Tutte's invariants technique: we are left with transforming the BVP of Proposition \ref{prop:BVP_Carleman} into a more classical one, using a certain conformal mapping having a very convenient gluing property.

\subsection{Conformal gluing}
\label{sec:gluing}

Our main result in this section is to prove that the function $W$ defined by \textcolor{black}{Equation \eqref{eq:definition_W} below} satisfies the properties of Lemma \ref{lem:conformal_gluing}, allowing to transform the Carleman BVP with shift on the curve $\R$ of Proposition \ref{prop:BVP_Carleman} into a classical BVP on the segment $[0,1]$.

\begin{figure}[hbtp]
\centering
\includegraphics[scale=0.9]{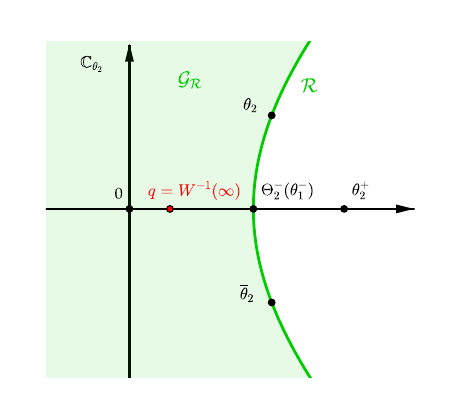}
\hspace{-1.6cm}
\includegraphics[scale=0.9]{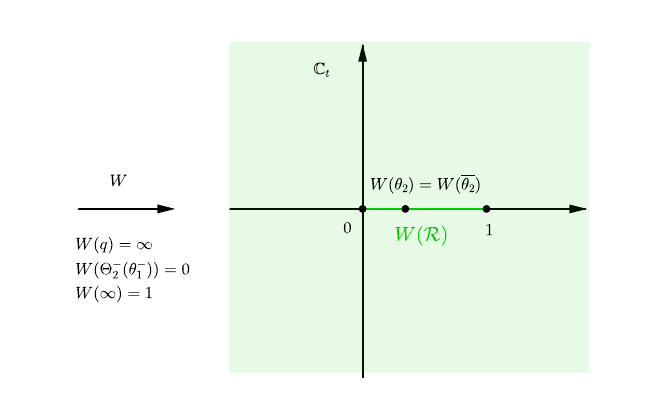}
\caption{Domains, curves and points related to the Carleman BVP with shift on $\R$ (left) and the standard BVP  on $[0,1]$ (right)}
\label{fig:different_levels}
\end{figure}

First we need to define $q$ by
\begin{equation}
\label{eq:def_q}
     q=\left\{\begin{array}{ll}
     p & \text{if } \gamma_1(\theta_1^-,\Theta_2^-(\theta_1^-))>0, \text{ i.e., if $\phi_1$ admits $p\in\G$ as a pole},\smallskip\\
     \frac{1}{2}\Theta_2^-(\theta_1^-) & \text{otherwise}.
     \end{array}\right. 
\end{equation}     
Note, the choice $\frac{1}{2}\Theta_2^-(\theta_1^-)$ is \textcolor{black}{arbitrary}: any point in $\G$ would have been suitable. See Figures~\ref{fig:p_ellipse}, \ref{fig:BVP_theta} \textcolor{black}{and \ref{fig:different_levels}}. \textcolor{black}{In the case where condition \eqref{eq:drift_negative} is not satisfied, $\Theta_2^-(\theta_1^-)$ may be negative and another choice for $q$ should be done (as for example $\Theta_2^-(\theta_1^-)-1$), see Section \ref{subsec:generalizations}.}

The function $W$ is built on the function $w$ below (note, $w$ is introduced in \cite[Theorem 1]{franceschi_tuttes_2016}; under the symmetry conditions $\mu_1 = \mu_2$, $\sigma_{11}=\sigma_{22}$, and symmetric reflection vectors in \eqref{eq:RBMQP}, Foschini \cite{Foschini} also obtained an expression for the conformal mapping $w$, see \cite[Figure 3]{Foschini}; see finally \cite[Equation (4.6)]{Baccelli_Fayolle_1987} for a related formulation of $w$):
\begin{equation*}
     w(\theta_2)=T_{\frac{\pi}{\beta}}\left(-\frac{2\theta_2-(\theta_2^+ +\theta_2^-)}{\theta_2^+ -\theta_2^-}\right),
\end{equation*}
which itself uses the branch points \eqref{eq:definition_theta_pm}, the generalized Chebyshev polynomial \eqref{eq:Chebyshev_polynomial} and the angle
\begin{equation}
%\label{eq:definition_beta}
     \beta= \arccos {-\frac{\sigma_{12}}{\sqrt{\sigma_{11}\sigma_{22}}}}.
\end{equation} 
\textcolor{black}{By \cite[Section 4.2]{franceschi_tuttes_2016}, $w$ is analytic on the cut plane $\mathbb C\setminus [\theta_2^+,\infty)$.} Then we define
\begin{equation}
\label{eq:definition_W}
     W(\theta_2)=\frac{w(\theta_2)-w(\Theta_2^-(\theta_1^-))}{w(\theta_2)-w(q)}=\frac{w(\theta_2)+1}{w(\theta_2)-w(q)}.
\end{equation}
\textcolor{black}{The last equality is due to the identity 
\begin{equation*}
     \Theta_2^-(\theta_1^-)=\frac{\theta_2^+ +\theta_2^-}{2}-\frac{\theta_2^+ -\theta_2^-}{2} \cos \beta. 
\end{equation*}
We can see it by a direct computation or using a uniformization of the zero set of the kernel, see \cite[Section 5]{franceschi_tuttes_2016} for more details. Then we have $w(\Theta_2^-(\theta_1^-))= T_{\frac{\pi}{\beta}}(\cos \beta)=\cos \pi =-1$.}

\begin{lem}
\label{lem:conformal_gluing}
The function $W$ in \eqref{eq:definition_W}
\begin{enumerate}[label={\rm (\roman{*})},ref={\rm (\roman{*})}]
     \item\label{item:conformal_1} is analytic in $\G \setminus\{q\}$, continuous in $\overline{\G}\setminus\{q\}$ and bounded at infinity,
     \item\label{item:conformal_2} is one-to-one from $\G\setminus\{q\}$ onto $\mathbb{C}\setminus[0,1]$,
     \item\label{item:conformal_3} satisfies  $W(\theta_2)=W(\overline{\theta_2})$ for all $\theta_2\in\mathcal{R}$.
\end{enumerate}
\end{lem}

\begin{proof}
It can be found in \cite[Lemma 6]{franceschi_tuttes_2016} that $w$ in \eqref{eq:definition_w}
\begin{enumerate}[label={\rm (\roman{*}')},ref={\rm (\roman{*}')}]
     \item\label{item:conformal_1'} is analytic in $\G$, continuous in $\overline{\G}$ and unbounded at infinity ($T_a$ admits an analytic continuation on $\mathbb C\setminus \textcolor{black}{(\theta_2^+,\infty)}$, and even on $\mathbb C$ if $a$ is a non-negative integer: \textcolor{black}{in that case $T_a$ is the classical Chebyshev polynomial of the first kind}), 
     \item\label{item:conformal_2'} is one-to-one from $\G$ onto $\mathbb{C}\setminus(-\infty,-1]$,
     \item\label{item:conformal_3'} satisfies $w(\theta_2)=w(\overline{\theta_2})$ for all $\theta_2\in\mathcal{R}$.
\end{enumerate}
Here we want to define another conformal gluing function, which glues together the upper part and the lower part of the hyperbola onto the segment $[0,1]$, and which sends the point $q$ in \eqref{eq:def_q} at infinity, see Figure \ref{fig:different_levels}. For this reason we set $W$ as in \eqref{eq:definition_W}: by construction $W(\Theta_2^-(\theta_1^-))=0$, $W(\infty)=1$ and $W(q)=\infty$. 
%Properties \ref{item:conformal_1},\ref{item:conformal_2},\ref{item:conformal_3} directly follow from \ref{item:conformal_1'},\ref{item:conformal_2'},\ref{item:conformal_3'}.
%
% \textcolor{black}{Ici, revoir les $p$ et $q$, et ajuster la fin de la preuve} 
The proof of Lemma \ref{lem:conformal_gluing} follows from the above-mentioned properties \ref{item:conformal_1'}--\ref{item:conformal_3'} of $w$ together with the definition \eqref{eq:definition_W} of $W$. 
\end{proof}

\begin{rem}
\label{eq:algebraic_nature_w_W}
The algebraic nature of the mapping $w$ in \eqref{eq:definition_w} (or equivalently $W$ in \eqref{eq:definition_W}) is directly related to the rationality of $\frac{\beta}{\pi}$. Precisely, as shown in \cite[Proposition 13]{franceschi_tuttes_2016}, the following behaviors are possible:
\begin{itemize}
     \item \textcolor{black}{The function $w$ is algebraic if and only if $\frac{\beta}{\pi}\in\mathbb Q$;}
     \item If in addition $\frac{\pi}{\beta}\in\mathbb N$ (and only in this case), then $w$ is a polynomial.
\end{itemize}     
\end{rem}

\subsection{Reduction to a standard BVP}
\label{sec:Riemann_BVP}

Thanks to the gluing function $W$ in \eqref{eq:definition_W}, we are able to reformulate the Carleman BVP as a standard BVP for an open contour. See Figure \ref{fig:different_levels} for a compact view of the two complex planes associated to the Carleman's and Riemann's BVPs. Define $\psi_1$ by 
\begin{equation} 
\label{eq:expression_psi_1}
     \psi_1 (t)=\phi_1 \circ W^{-1} (t),\qquad \forall t\in\mathbb{C}\setminus [0,1]
\end{equation}
(note, $\psi_1$ is meromorphic on $\mathbb{C}\setminus [0,1]$). Equivalently we have $\phi_1 (\theta_2)=\psi_1 \circ W (\theta_2)$ for $\theta_2\in \G$. Obviously $W^{-1}$ is not well defined on $[0,1]$; however, it does admit upper and lower limits for $t\in [0,1]$:
\begin{equation*}
     (W^{-1})^+(t) =\lim_{\substack{u\to t\\\Im u>0}}W^{-1}(u),\qquad
     (W^{-1})^-(t)  =\lim_{\substack{u\to t\\\Im u<0}}W^{-1}(u),
\end{equation*}
and similarly for $\psi_1^+(t)$ and $\psi_1^-(t)$. Then for $\theta_2 \in \mathcal{R}$ and $t=W(\theta_2)=W(\overline{\theta_2})$, we have 
\begin{equation*}
\phi_1({\theta_2})=
\left\{\begin{array}{ll}
\psi_1^+(t) & \text{if } \Im\,\theta_2>0, \smallskip\\
\psi_1^-(t) & \text{if } \Im\,\theta_2<0,
\end{array}\right.\qquad
\phi_1(\overline{\theta_2})=
\left\{\begin{array}{ll}
\psi_1^-(t) & \text{if } \Im\,\theta_2>0, \smallskip\\
\psi_1^+(t) & \text{if } \Im\,\theta_2<0.
\end{array}\right.
\end{equation*}
Define further
\begin{equation} 
\label{def:H}
     H(t)= G((W^{-1})^-(t)),\qquad \forall t\in[0,1].
\end{equation}
Then Proposition \ref{prop:BVP_Carleman} becomes:
\begin{prop}[Riemann BVP]
\label{prop:BVP_Riemann}
The function $\psi_1$ in \eqref{eq:expression_psi_1}
\begin{enumerate}[label={\rm (\arabic{*})},ref={\rm (\arabic{*})}]
     \item\label{item:BVPRiemann1}
      is analytic in $\mathbb{C}\setminus [0,1]$, bounded at infinity if $\gamma_1(\theta_1^-,\Theta_2^-(\theta_1^-))\leqslant 0$ %(with limits $\phi_1(q)$) 
     and admitting a simple pole at infinity otherwise,
     \item\label{item:BVPRiemann2} is continuous on $[0,1]$ from below (with limits $\psi_1^-$) and above (with limits $\psi_1^+$),
bounded at $0$ %(equals to $\phi_1 (\Theta_2^-(\theta_1^-))$), 
%bounded at 
and $1$ %(equals to $\lim_{\infty} \phi_1 =0$ (\`a v\'erifier))
(except if $\gamma_1(\theta_1^-,\Theta_2^-(\theta_1^-)) = 0$: in this case it has a pole of order one at $0$),
     \item\label{item:BVPRiemann3} satisfies, with $H$ defined in \eqref{def:H}, the boundary condition
\begin{equation} 
\label{eq:condition_Riemann_boundary}
     \psi_1^+(t)=H(t)\psi_1^-(t),\qquad \forall t\in[0,1].
\end{equation}
\end{enumerate}
\end{prop}
\begin{proof}
Items \ref{item:BVPRiemann1} and \ref{item:BVPRiemann2} directly follow from the corresponding items in Proposition \ref{prop:BVP_Carleman}. 
With the above definitions, the boundary equation \eqref{eq:bound_cond_gen} becomes
\begin{equation*}
     \left\{\begin{array}{ll}
     \psi_1^- (t)= G(\theta_2)\psi_1^+(t)  
&\text{if } \Im\,\theta_2
>0,\smallskip\\
\psi_1^+ (t)= G(\theta_2)\psi_1^-(t)
 &\text{if } \Im\,\theta_2
<0.
\end{array}\right.
\end{equation*}
Since
$
\frac{1}{G(\theta_2)}=G(\overline{\theta_2})=
H(t)$ if $\Im\,\theta_2>0$, and $G({\theta_2})=
H(t)$ if $\Im\,\theta_2<0$, the last item follows.
\end{proof}

\subsection{Index of the BVP}
\label{sec:index}

The resolution of BVPs as in Proposition \ref{prop:BVP_Riemann} heavily depends on the \textit{index} $\chi$ (see, e.g., \cite[Section 5.2]{litvinchuk_solvability_2000}), which is related to the variation of argument of $H$ on $[0,1]$:
\begin{equation}
\label{eq:def_index}
     \Delta=[\arg H ]_0^1, \qquad
     {d}=\arg H(0) \in (-\pi,\pi], \qquad
     \chi = \left\lfloor \frac{{d}+\Delta}{2\pi}\right\rfloor.
\end{equation}
$\Delta$ quantifies the variation of argument of $H$ on $[0,1]$ and $\arg H(1)={d}+\Delta$. Since $(W^{-1})^- ([0,1])=\R^-$, $\Delta$ in \eqref{eq:def_index} can be equivalently written $[\arg G]_{\R^-}$ (from the vertex to infinity). 

\begin{rem}
\label{rem:delta_arbitrary}
It is important to notice that ${d}\in(-\pi,\pi]$ in \eqref{eq:def_index} corresponds to an arbitrary choice. Any other choice would eventually lead to the same Theorem \ref{thm:main} (though written slightly differently).
\end{rem}

First, we compute ${d}$ in \eqref{eq:def_index}.
\begin{lem}
\label{lem:delta}
We have
\begin{equation}
\label{eq:definition_delta}
{d}=
     \left\{\begin{array}{ll}
     0 
&\text{if } \gamma_1(\theta_1^-,\Theta_2^-(\theta_1^-)) \neq 0
,\smallskip\\
\pi
&\text{if }
\gamma_1(\theta_1^-,\Theta_2^-(\theta_1^-)) = 0.
\end{array}\right.
\end{equation}
The angle ${d}+\Delta\in (-2\pi,2\pi)$ and we have
\begin{equation}
\label{eq:tan_Delta}
     \tan\frac{{d}+\Delta}{2} = \frac{\det R\cdot \textcolor{black}{\sqrt{\det\Sigma}} }{\sigma_{12}(r_{11}r_{22}+r_{12}r_{21})-\sigma_{22}r_{11}r_{12}-\sigma_{11}r_{22}r_{21}}.
\end{equation}
\end{lem}
Note that the denominator of \eqref{eq:tan_Delta} can be negative, zero or positive, depending on the parameters. 

\begin{proof}[Proof of Lemma \ref{lem:delta}]
\textcolor{black}{First of all we show the formula \eqref{eq:definition_delta}.
If first $\gamma_1(\theta_1^-,\Theta_2^-(\theta_1^-)) \neq 0$ then $H(0)=G(\Theta_2^-(\theta_1^-))=1$, since $\Theta_2^-(\theta_1^-)\in \mathbb{R}$ simplifies the quotient \eqref{eq:definition_G} and then ${d}=\arg H(0)=0$. In the other case $\gamma_1(\theta_1^-,\Theta_2^-(\theta_1^-)) = 0$, and we have %limit of $H$ at $0$ is equal to 
$$
\lim_{t\to 0} H(t)=
\lim_{\theta_2 \to \Theta_2^-(\theta_1^-) \atop \theta_2\in \R^-}
\frac{\theta_2-\Theta_2^-(\theta_1^-)}{\overline{\theta_2}-\Theta_2^-(\theta_1^-)} =-1.$$ 
The last equality is due to the fact that the tangent to $\R$ at $\Theta_2^-(\theta_1^-)$ is vertical, see Figure \ref{fig:different_levels}.
Indeed if we write $\theta_2-\Theta_2^-(\theta_1^-)=a+ib$ when $\theta_2 \to \Theta_2^-(\theta_1^-)$ with $ \theta_2\in \R^-$, the vertical tangent gives $\frac{a}{b}\to 0$ and then $\lim_{t\to 0} H(t)=\lim_{\frac{a}{b}\to 0} \frac{\frac{a}{b}+i}{\frac{a}{b}-i}=-1$. It implies that ${d}=\pi$.}

\textcolor{black}{We are now going to show \eqref{eq:tan_Delta}.} We start by remarking that for $\theta_2\in\R$, $G(\theta_2)=1$ if and only if $\theta_2 \in \mathbb{R}$. Accordingly, for $t\in [0,1]$, $H(t)=1$ only at $t=0$. Since $\vert H\vert=1$ on $[0,1]$, then necessarily ${d}+\Delta \in [-2\pi,2\pi]$. We now calculate 
\begin{equation*}
     H(1)=\lim_{ \theta_2 \to \infty\hspace{1mm} \atop \theta_2\in\R^-} G(\theta_2).
\end{equation*}
\textcolor{black}{Thanks to Equation \eqref{eq:definition_Theta_pm} we easily compute the following limit}
\begin{equation*}
     \lim_{ \theta_2 \to \infty\hspace{1mm} \atop \theta_2\in\R^-}\frac{\Theta_1^-(\theta_2)}{\theta_2}=\frac{-\sigma_{12}-i\sqrt{\det\Sigma}}{\sigma_{11}}.
\end{equation*}
Using this limit together with the definition of $G$ (see \eqref{eq:definition_G})
\begin{equation*}
G(\theta_2) =\frac{\big(r_{11}\frac{\Theta_1^-(\theta_2)}{\theta_2}+r_{21}\big)\big(r_{12}\frac{\Theta_1^-(\theta_2)}{\overline{\theta_2}}+r_{22}\big)}{\big(r_{12}\frac{\Theta_1^-(\theta_2)}{\theta_2}+r_{22}\big)\big(r_{11}\frac{\Theta_1^-(\theta_2)}{\overline{\theta_2}}+r_{21}\big)},
\end{equation*}
we obtain that
\begin{align*}
H(1)&=
\frac{\big(r_{11}(-\sigma_{12}-i\sqrt{\det\Sigma})+r_{21}\sigma_{11}\big)\big(r_{12}(-\sigma_{12}+i\sqrt{\det\Sigma})+r_{22}\sigma_{11}\big)}{\big(r_{12}(-\sigma_{12}-i\sqrt{\det\Sigma})+r_{22}\sigma_{11}\big)\big(r_{11}(-\sigma_{12}+i\sqrt{\det\Sigma})+r_{21}\sigma_{11}\big)}
\\
&=
\frac{\sigma_{22}r_{11}r_{12}+\sigma_{11}r_{22}r_{21}-\sigma_{12}(r_{11}r_{22}+r_{12}r_{21})-i\det R\textcolor{black}{\sqrt{\det\Sigma}}}{\sigma_{22}r_{11}r_{12}+\sigma_{11}r_{22}r_{21}-\sigma_{12}(r_{11}r_{22}+r_{12}r_{21})+i\det R \textcolor{black}{\sqrt{\det\Sigma}}}=\exp({i({d}+\Delta)}).
\end{align*}
\textcolor{black}{Remembering that $\arg H(1)={d}+\Delta\in [-2\pi,2\pi]$,} it gives \eqref{eq:tan_Delta} and clearly, ${d}+\Delta$ cannot be equal to $\pm2\pi$ because $\det R \cdot \det \Sigma \neq 0$.
\end{proof}
\begin{comment}
\textcolor{black}{Dans la preuve du Lemme \ref{lem:delta}, a-t-on int\'er\^et \`a faire une interpr\'etation avec $s(\rho)$, et \`a reproduire l'angle sur le Figure \ref{fig:ellipse_uniformization}?}
\end{comment}
We now prove that
\begin{equation}
\label{eq:def_chi}
     \chi=\left\{\begin{array}{rl}
     0 & \displaystyle\text{if } \gamma_1(\theta_1^-,\Theta_2^-(\theta_1^-))\leqslant 0,\smallskip\\
      -1 & \displaystyle\text{if } \gamma_1(\theta_1^-,\Theta_2^-(\theta_1^-))>0.
      \end{array}\right.
\end{equation}  
In particular, $\chi$ is an intrinsically non-continuous function of the parameters, \textcolor{black}{see also Remark \ref{rem:different_cases_singularities}}.

\textcolor{black}{Recall that the function $\psi_1$ has been defined in \eqref{eq:expression_psi_1}.}
\begin{lem} 
\label{lem:index}
The index $\chi$ can take only the values $0$ and $-1$, and we have the dichotomy:
\begin{itemize}
     \item $\chi=0\phantom{-} \Longleftrightarrow \gamma_1(\theta_1^-,\Theta_2^-(\theta_1^-)) \leqslant 0 \Longleftrightarrow $ $\psi_1$ has no pole at infinity, %nor at $0$,
\item $\chi=-1 \Longleftrightarrow \gamma_1(\theta_1^-,\Theta_2^-(\theta_1^-))> 0 \Longleftrightarrow $ $\psi_1$ has a simple pole at infinity. %or at $0$.
\end{itemize}
%In the second case, the pole is of order $1$, $\displaystyle\psi(z) \mathop{\sim}_{\infty} z$.
\end{lem}
Note that a simple pole at infinity means that $\displaystyle\psi_{1}(t) \mathop{\sim}_{\infty} c\cdot t$ for some non-zero constant $c$.
%This lemma is compatible with remark p102 of \cite{fayolle_random_1999} on the link between $\chi$ and the order at infinity. \textcolor{black}{$\leftarrow$ Enlever cette remarque?}

\begin{proof}
We have already seen in Proposition \ref{prop:BVP_Carleman} that the sign of $\gamma_1(\theta_1^-,\Theta_2^-(\theta_1^-))$ determines whether $\phi_1$ has a pole in $\G$ or not, and thus if $\psi_1$ has a pole at infinity \textcolor{black}{by the correspondence of Figure~\ref{fig:different_levels}}. %(or at $0$ when the pole of $\phi_1$ is on $\R$). 
This shows the two  equivalences on the right in the statement of Lemma \ref{lem:index}. 
We are thus left with proving the first two equivalent conditions. 

First, if $\gamma_1(\theta_1^-,\Theta_2^-(\theta_1^-))=0$, ${d}=\pi$ and we have seen that in this case $H(t)\neq 1$ for all $t\in[0,1]$. By \eqref{eq:def_index}, we deduce that $\chi=0$.

If now $\gamma_1(\theta_1^-,\Theta_2^-(\theta_1^-))\neq 0$ we notice that
\begin{equation*}
     \chi = \left\lfloor \frac{\Delta}{2\pi}\right\rfloor =0 \text{ or } -1.
\end{equation*}
Indeed, we have proved in Lemma \ref{lem:delta} that $\Delta \in (-2\pi,2\pi)$. In particular, the sign of $\Delta$ determines $\chi$: if $\sgn \Delta \geq 0$ then $\chi=0$ and if $\sgn \Delta < 0$, $\chi=-1$. In the rest of the proof, we show that $\sgn \Delta=-\sgn \gamma_1(\theta_1^-,\Theta_2^-(\theta_1^-))$. First, $\Delta$ can be computed as
\begin{equation*}
     \Delta = \arg H(1)= [\arg G ]_{\R^-}=\left[\arg \frac{\gamma_1\overline{\gamma_2}}{\gamma_2 \overline{\gamma_1}} (\Theta_1^-(\theta_2),\theta_2)\right]_{\R^-}.
\end{equation*}
Let $\theta_2=a-ib \in \R^-$ (we must have $b\geq0$ and $a \geq \Theta_2^-(\theta_1^-)>0$, see Figure \ref{fig:BVP_theta}) and $\theta_1= \Theta_1^-(\theta_2) \in(-\infty,\theta_1^-]$. Using the expression \eqref{eq:def_gamma_gamma1_gamma2} of $\gamma_1$ and $\gamma_2$, we obtain 
\begin{equation*}
     \gamma_1 \overline{\gamma_2} (\theta_1,\theta_2) = 
\gamma_1(\theta_1,a)\gamma_2(\theta_1,a) + r_{21}r_{22}b^2+ib\theta_1\det R,
\end{equation*} 
from where we deduce that
\begin{equation}
\label{eq:arg_to_compute}
     \arg \frac{\gamma_1 \overline{\gamma_2}}{\gamma_2 \overline{\gamma_1}} (\theta_1,\theta_2)  = 2\arctan \frac{b\cdot\theta_1\cdot\det R}{\gamma_1(\theta_1,a)\gamma_2(\theta_1,a) + r_{21}r_{22}b^2}.
\end{equation}
We now look for the sign of \eqref{eq:arg_to_compute} when $\theta_2 \to \Theta_2^-(\theta_1^-)$, while remaining in $\R^-$. This is sufficient to give the sign of $\Delta$, because \eqref{eq:arg_to_compute} does not change sign on $\R^-$ due to the fact that $G(\theta_2)=1$ on $\R^-$ if and only if $\theta_2=\Theta_2^-(\theta_1^-)$.

When $\theta_2 \to \Theta_2^-(\theta_1^-)$ we have $b\to 0$, $a\to \Theta_2^-(\theta_1^-)$ and $\theta_1\to \theta_1^-$. We thus have
\begin{align*}
     \sgn \arg \frac{\gamma_1 \overline{\gamma_2}}{\gamma_2 \overline{\gamma_1}} (\theta_1,\theta_2)  &= \sgn b\cdot\sgn \theta_1\cdot \sgn \det R\cdot \sgn \gamma_2(\theta_1^-,\Theta_2^-(\theta_1^-)) \cdot\sgn \gamma_1(\theta_1^-,\Theta_2^-(\theta_1^-))
\\
&=
(+1)(-1)(+1)(+1) \sgn \gamma_1(\theta_1^-,\Theta_2^-(\theta_1^-))\\&=-\sgn \gamma_1(\theta_1^-,\Theta_2^-(\theta_1^-)),
\end{align*}
because $b\geq0$, $\theta_1<0$, $\det R>0$ by \eqref{eq:stationary_distribution_CNS}, and $\gamma_2(\theta_1^-,\Theta_2^-(\theta_1^-))>0$ (because $r_{22}>0$ and $r_{22}\mu_1-r_{12}\mu_2<0$, see Figure \ref{fig:p_ellipse} to visualize this geometric condition).
Then $\sgn \Delta =\sgn \arg \frac{\gamma_1 \overline{\gamma_2}}{\gamma_2 \overline{\gamma_1}} (\theta_1,\theta_2) =-\sgn \gamma_1(\theta_1^-,\Theta_2^-(\theta_1^-))$, concluding the proof.
 \end{proof}

\subsection{Resolution of the BVP}
\label{sec:resolution}
\textcolor{black}{We are now in position to conclude the proof of Theorem \ref{thm:main}.}
Reformulating Proposition \ref{prop:BVP_Riemann}, the function $\psi_1$ in \eqref{eq:expression_psi_1}
\begin{itemize}
     \item is sectionally analytic in $\mathbb{C}\setminus [0,1]$,
     \item is continuous on $[0,1]$ from below (with limits $\psi_1^-$) and above (with limits $\psi_1^+$), 
     \item[]
\begin{itemize}
\item[$\bullet$]     
     is bounded at the vicinities of $[0,1]$ if $\gamma_1(\theta_1^-,\Theta_2^-(\theta_1^-)) \neq 0$, %(then taking the value $\phi_1 (\Theta_2^-(\theta_1^-))$ and $\lim_{\infty} \phi_1 =0$, respectively) \textcolor{black}{est-ce utile de sp\'ecifer les valeurs exactes?}\textcolor{black}{oui c'est ce que je voulais dire, car je sais pas trop comment montrer rapidement que la limite est zero sinon}
     \item[$\bullet$]   has a pole of order one at $0$ and bounded at $1$ if $\gamma_1(\theta_1^-,\Theta_2^-(\theta_1^-)) = 0$,
     \end{itemize}
     \item is bounded at infinity if there is no pole before $\Theta_2^-(\theta_1^-)$ (then taking the value $\phi_1(q)$), and with a pole of order one (see Lemma \ref{lem:continuation}) at infinity if not (in short, it has a pole of order $-\chi$ at infinity),
     \item satisfies $\psi_1^+(t)=H(t)\psi_1^-(t)$ for $t\in[0,1]$, with index $\chi$ given by \eqref{eq:def_chi}, cf.\ also Lemma~\ref{lem:index}.
\end{itemize}

\begin{proof}[End of proof of Theorem \ref{thm:main}] 
Our main reference for the resolution of the above so-called homogeneous BVP on an open contour is the book \cite{Mu-72} of Muskhelishvili, see in particular \cite[\S 79]{Mu-72}.

First of all, we prove that there exists a non-zero constant $c$ such that
\begin{equation}
\label{eq:formula_psi_1}
     \psi_1(t)=c(t-1)^{-\chi}\exp\Gamma(t),
\end{equation}
where $\Gamma$ is the following function, sectionally analytic on $\mathbb{C}\setminus [0,1]$:
\begin{equation}
\label{eq:formula_Gamma}
     \Gamma(t)=\frac{1}{2i\pi} \int_{0}^{1} \frac{\log H(z)}{z-t} \mathrm{d}z.
\end{equation}
To make precise the definition \eqref{eq:formula_Gamma}, we define $\log H(z)$ by the facts that it should vary continuously over $[0,1]$, and its initial value is such that $\log H(0)=i{d}$ (i.e., $0$ if $H(0)=1$ and $i\pi$ if $H(0)=-1$, see \eqref{eq:definition_delta}).

At the vicinities $0$ and $1$, we have by \cite[\S 29 and \S 79]{Mu-72} that
\begin{equation}
\label{eq:formula_expGamma}
     \exp{\Gamma(t)}=t^{-\frac{{d}}{2\pi}} \Omega_0 (t),
     \qquad
     \exp{\Gamma(t)}=(t-1)^{\frac{{d}+\Delta}{2\pi}} \Omega_1 (t),
\end{equation}
for some function $\Omega_{0}$ (resp.\ $\Omega_{1}$) analytic in a neighborhood of $0$ (resp.\ $1$) and non-zero at $0$ (resp.\ $1$). Then we set 
\begin{equation}
\label{eq:def_X}
     X(t)=t^{\frac{{d}}{\pi}}(t-1)^{-\chi}\exp{\Gamma(t)}.
\end{equation}
The function $X$ in \eqref{eq:def_X} is sectionally analytic in $\mathbb{C}\setminus [0,1]$, and by construction of $\Gamma$ and the Sokhotski-Plemelj formulas, it satisfies the boundary condition \eqref{eq:condition_Riemann_boundary} (see \cite[\S 79]{Mu-72} for more details). Furthermore it has a pole of order $-\chi+\frac{{d}}{\pi}$ at infinity and is bounded at $0$ and $1$: indeed, $\frac{{d}}{\pi}$ and $-\chi$ are both equal to $0$ or $1$, see Lemmas \ref{lem:delta} and \ref{lem:index}. Then we consider two cases separately. 

\smallskip

$\bullet$ First case: $\gamma_1(\theta_1^-,\Theta_2^-(\theta_1^-)) \neq 0$. Then ${d}=0$, and the function $X$ in \eqref{eq:def_X} simplifies into
\begin{equation*}
     X(t)=(t-1)^{-\chi}\exp{\Gamma(t)}.
\end{equation*}
It satisfies the exact same boundary condition as \eqref{eq:condition_Riemann_boundary}. %the entire BVP of Proposition \ref{prop:BVP_Riemann} \textcolor{black}{$\leftarrow$ un peu maladroit de dire \c ca, car justement on ne v\'erifie pas tous les points de la Proposition \ref{prop:BVP_Riemann}, par exemple pas les valeurs exactes. Ne vaudrait-il pas mieux \'evoquer seulement la condition au bord \eqref{eq:condition_Riemann_boundary}?}. 
Looking at the ratio $\frac{\psi_1}{X}$, the boundary condition \eqref{eq:condition_Riemann_boundary} gives that on $[0,1]$,
\begin{equation*}
     \frac{\psi_1^+}{X^+}=\frac{\psi_1^-}{X^-}.
\end{equation*}
The above ratio is then analytic in the entire plane, even at the vicinities $0$ and $1$. The point $0$ is a regular point and $1$ is a removable singularity. Indeed, $1$ is an isolated singular point, at which $\frac{\psi_1}{X}$ might be infinite with degree less than unity (namely, $-\chi +\frac{\Delta}{2\pi}$). Moreover, the function $\frac{\psi_1}{X}$ is bounded at infinity, because both $X$ and $\psi_1$ have a pole of the same order $-\chi$. Thanks to Liouville's theorem, we deduce that $\frac{\psi_1}{X}$ is constant. In conclusion, the formula \eqref{eq:formula_psi_1} holds in this case.

\smallskip

$\bullet$ Second case: $\gamma_1(\theta_1^-,\Theta_2^-(\theta_1^-)) = 0$. Then ${d}=\pi$, $\chi=0$ and $X(t)=t\exp{\Gamma(t)}$ in \eqref{eq:def_X}. Firstly, we notice that the function $t\psi_1$ satisfies the boundary condition \eqref{eq:condition_Riemann_boundary}, is bounded at $0$ and $1$, and has a pole of order one at infinity. Moreover, the function $X$ has a pole of order $1$ at infinity. Considering then the ratio $\frac{t\psi_1}{X}$, the boundary condition \eqref{eq:condition_Riemann_boundary} implies that on $[0,1]$,
\begin{equation*}
     \frac{t\psi_1^+}{X^+}=\frac{t\psi_1^-}{X^-}.
\end{equation*}
The above ratio function is thus analytic in the entire complex plane, including the vicinities $0$ and $1$. It is indeed bounded at $1$, and has a removable singularity at $0$: the point $0$ is an isolated singular point, at which $\frac{t\psi_1}{X}$ might be infinite with degree less than $\frac{1}{2}$. Using again Liouville's theorem, we deduce that the function $\frac{t\psi_1}{X}$ is a constant. Formula \eqref{eq:formula_psi_1} therefore also holds.
%\begin{itemize}
%\item if $\chi=-1$ it is bounded at infinity and we deduce it is a constant. Then up to a constant $\psi_1(z)$ is equal to $\frac{X(z)}{z}=\frac{z-1}{z}e^{\Gamma(z)}$
%\item if $\chi=0$ it has a pole of order $1$ at infinity and we deduce that up to a constant and for some $c$, the function $\frac{z\psi_1(z)}{X(z)}$ is equal to $z+c$. If we notice that $X$ has no pole and that $\psi_1$ has no zero exept $1$ (because $\phi_1$ never cancel exept at infinity \textcolor{black}{(pourquoi?)}) we deduce that $c=-1$. Then up to a constant $\psi_1(z)$ is still equal to $\frac{X(z)}{z}=\frac{z-1}{z}e^{\Gamma(z)}$.
%\end{itemize}

\smallskip

We now deduce from \eqref{eq:formula_psi_1} the formula \eqref{eq:main_formula_with_constants} stated in Theorem \ref{thm:main}. Going from the $t$-plane back to the $\theta_2$-plane (see \eqref{eq:expression_psi_1} \textcolor{black}{and Figure \ref{fig:different_levels}}), one has that for some constant $c$,
\begin{equation}
\label{eq:Wmainformula}
\phi_1(\theta_2)=\psi_1(W(\theta_2))
=c(W(\theta_2)-1)^{-\chi} \exp\bigg\{\frac{1}{2i\pi} \int_{\R^-} \log G(\theta)\frac{W'(\theta)}{W(\theta)-W(\theta_2)}\mathrm{d}\theta\bigg\}.
\end{equation}
%and when $\gamma_1(\theta_1^-,\Theta_2^-(\theta_1^-)) = 0$
%we have for some constant $C$
%$$
%\phi_1(\theta_2)=\psi_1(W(\theta_2))
%=C \frac{(W(\theta_2)-1)}{W(\theta_2)} \exp\bigg\{\frac{1}{2i\pi} \int_{\R^-} \log G(\theta)\frac{W'(\theta)}{W(\theta)-W(\theta_2)}\mathrm{d}\theta\bigg\}.
%$$
Using the equation \eqref{eq:definition_W} relating $W$ and $w$, we easily obtain 
\begin{equation*}
     W(\theta_2)-1=\frac{w(q)+1}{w(\theta_2)-w(q)} 
\end{equation*}
as well as
\begin{equation*}
     \frac{W'(\theta)}{W(\theta)-W(\theta_2)}=\frac{w'(\theta)}{w(\theta)-w(\theta_2)} - \frac{w'(\theta)}{w(\theta)-w(q)}.
\end{equation*}
Remembering that in the case $\chi=-1$ one has $q=p$, see \eqref{eq:def_q}, we finally obtain that for some constant $c'$, 
\begin{equation*}
     \phi_1(\theta_2)=c'\left(\frac{1}{w(\theta_2)-w(p)}\right)^{-\chi} \exp \bigg\{\frac{1}{2i\pi} \int_{\R^-} \log G(\theta)\frac{w'(\theta)}{w(\theta)-w(\theta_2)}\mathrm{d}\theta \bigg\}.
\end{equation*}     
By definition \eqref{eq:Laplace_transform_boundary} of the Laplace transform we have $\phi_1(0)=\nu_1(\mathbb{R_+})$. To find the exact constant $c'$ (and thereby our main result \eqref{eq:main_formula_with_constants}), we simply evaluate the above formula at $\theta_2=0$ \textcolor{black}{and use Lemma \ref{lem:value_at_0} below}.
\end{proof}
\textcolor{black}{
\begin{lem}
\label{lem:value_at_0}
One has $\left(  \begin{array}{r} \phi_1(0) \\  \phi_2(0)\end{array} \right) = \left(  \begin{array}{r} \nu_1(\mathbb{R}_+) \\  \nu_2(\mathbb{R}_+)\end{array} \right) =- R^{-1} \mu$.
\end{lem}
\begin{proof}
Equation \eqref{eq:Laplace_transform_boundary} evaluated at $0$ gives $\phi_i(0)=\nu_i(\mathbb{R}_+)$. We now evaluate \eqref{eq:functional_equation} at $\theta_2=0$, divide by $\theta_1$ and finally evaluate at $\theta_1=0$. This yields $-\mu_1=r_{11}\phi_1(0)+r_{12}\phi_2(0)$. In a similar way, we obtain $-\mu_2=r_{21}\phi_1(0)+r_{22}\phi_2(0)$, thereby concluding the proof.
\end{proof}}

\textcolor{black}{Note that Lemma \ref{lem:value_at_0} gives, as announced in Theorem \ref{thm:main},
\begin{equation*}
     \nu_1(\mathbb{R}_+)=\frac{r_{12}\mu_2-r_{22}\mu_1}{\det R},\qquad
     \nu_2(\mathbb{R}_+)=\frac{r_{21}\mu_1-r_{11}\mu_2}{\det R},
\end{equation*}
which by \eqref{eq:stationary_distribution_CNS} are positive: \textcolor{black}{remind that this positivity condition is necessary for the existence of the stationary distribution \cite{harrison_brownian_1987}.}}

Clearly, the integral expression \eqref{eq:main_formula_with_constants} of $\phi_1$ is meromorphic in the domain $\G$. In Section \ref{sec:asymptotics} we shall see that it can be meromorphically continued on the much larger domain $\mathbb C\setminus [\theta_2^+,\infty)$.

To conclude this part, let us make Remark \ref{rem:delta_arbitrary} more precise. In the case $\gamma_1(\theta_1^-,\Theta_2^-(\theta_1^-)) = 0$ (i.e., $H(0)=-1$), we have chosen ${d}=\pi$ in \eqref{eq:def_index}. (Recall that choosing ${d}$ is tantamount to fixing a determination of the $\arg$ (or $\log$) function.) Remarkably, any other choice of ${d}$ would have led to the same explicit expression for $\phi_1$, though written differently. For instance, if we had taken ${d}=-\pi$ instead, the index $\chi$ would have been $-1$ (instead of $0$), and the two determinations of the logarithm would differ by $-2i\pi$. With our notation in the proof of Theorem \ref{thm:main}, we would have obtained for some constant $c''$,
\begin{equation*}
     \psi_1(t)=c'' \frac{t-1}{t}\exp{\Gamma(t)}. 
\end{equation*}
This actually corresponds to the formula of Theorem \ref{thm:main} associated with $\chi=-1$.

\subsection{Generalizations}
\label{subsec:generalizations}

Our main Theorem \ref{thm:main} is derived under the hypothesis \eqref{eq:drift_negative} that the coordinates $(\mu_1,\mu_2)$ of the drift are negative. However, the conditions \eqref{eq:stationary_distribution_CNS} (equivalent to the existence of a stationary distribution) allow the drift to have one non-negative coordinate. In the next few lines, we assume that $\mu_1\geq0$ or $\mu_2\geq0$, and we comment on the slight differences which would arise in the analytic treatment of the functional equation \eqref{eq:functional_equation}. 

\textcolor{black}{In the case of a drift having one non-negative coordinate, the drift vector is directed towards one axis (see Figures \ref{fig:tiger} and \ref{fig:mouse}) and accordingly, the pathwise behavior of the reflected Brownian motion is quite different. However, as we are going to explain now, the exact same integral formula stated in Theorem \ref{thm:main} still holds, and the only technical differences may be summarized as follows:}

\smallskip

\textcolor{black}{\textit{The ellipse:} in case of a drift with one non-negative coordinate, the ellipse $\gamma=0$ of Figure~\ref{fig:p_ellipse} would be oriented differently, because the drift is orthogonal to the tangent at the origin of the ellipse, see Figure \ref{fig:drift_negatif}. In particular, one may observe on the ellipse that the real point of the hyperbola $\Theta_2^{-}(\theta_1^-)$ may be negative. It is always negative when $\mu_1<0$ and $\mu_2\geq0$, while it may be positive, negative or zero when $\mu_1\geq0$ and $\mu_2<0$.}

\smallskip

\textcolor{black}{\textit{The continuation:} regarding the meromorphic continuation of  Lemma \ref{lem:continuation} and Lemma \ref{lem:cont}, there are the following changes:
\begin{itemize}
     \item In the case where $\mu_1<0$ and $\mu_2\geq0$ (Figure \ref{fig:tiger}), the real point $\Theta_2^{-}(\theta_1^-)$ is negative. If in addition $\sigma_{12}\leq0$ then the whole domain $\G$ has non-positive real part, and there is no need to continue $\phi_1$. On the other hand, if $\sigma_{12}>0$, then $\G\cap\{\theta_2\in\mathbb C:\Re\,\theta_2 > 0\}$ is composed of two connected components, in which one may continue $\phi_1$ with the same formula as in Lemma \ref{lem:continuation}. Whatever the sign of $\sigma_{12}$ is, the function $\phi_1$ has never a pole in the domain $\G$.
     \item \textcolor{black}{The case where $\mu_1\geq0$ and $\mu_2<0$ (Figure \ref{fig:mouse}) may be either similar to that of a double negative drift when the real point $\Theta_2^{-}(\theta_1^-)$ is positive, or similar to the case where $\mu_1<0$ and $\mu_2\geqslant 0$ when $\Theta_2^{-}(\theta_1^-)$ is negative.}
\end{itemize}}

\smallskip

\textcolor{black}{\textit{The BVP:} finally, the BVP (our Proposition \ref{prop:BVP_Carleman}) is exactly the same, and thus the explicit formula for the Laplace transform (Theorem \ref{thm:main}) also.}

\smallskip

\begin{figure}[hbtp]
\centering
\includegraphics[scale=0.8]{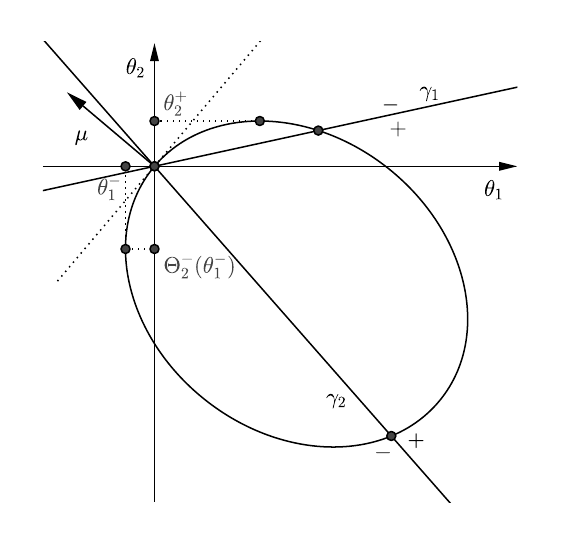}
\includegraphics[scale=0.8]{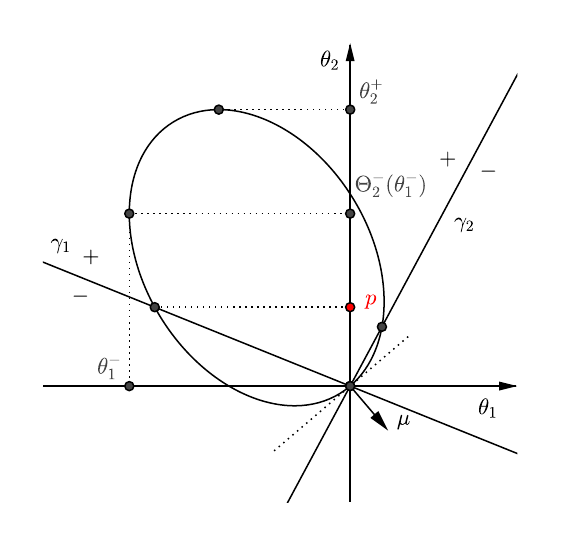}
%\caption{titre}
%\end{figure}
%
%\begin{figure}[hbtp]
%\centering
\includegraphics[scale=1]{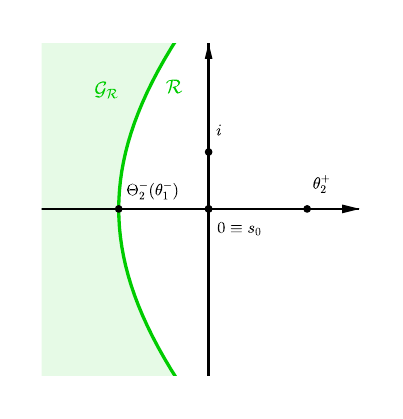}
\includegraphics[scale=1]{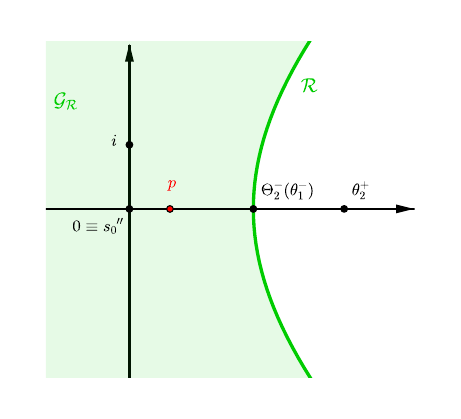}
\caption{\textcolor{black}{On the left $\mu_1<0$ and $\mu_2>0$, and on the right $\mu_1>0$ and $\mu_2<0$. The figures on the top represent the ellipses $\gamma=0$, and the figures on the bottom show the domains $\G$ and related quantities. Recall that the drift is orthogonal to the tangent at the origin of the ellipse. Remark that contrary to the doubly negative drift case, the real point $\Theta_2^{\pm} (\theta_1^-)$ of the hyperbola $\R$ may now be negative}}
\label{fig:drift_negatif}
\end{figure}

Let us also very briefly mention here the case of reflected Brownian motion in higher dimension \cite{HaRe-81,harrison_multidimensional_1987,BrDaHa-10,Br-11,DaHa-12}. Compared to its two-dimensional analogue, much less is known. However, an analogue of the functional equation \eqref{eq:functional_equation} can still be stated (since the BAR exists in any dimension, see \cite{harrison_brownian_1987,dai_reflected_1992}). Clearly, our techniques (based on complex analysis) use the dimension $2$, and in our opinion, generalizing in higher dimension these BVP techniques is a difficult open problem. In the discrete case too, the case of dimension $3$ is less understood. One can mention \cite{Co-84} for some ideas to state a BVP in dimension $3$, as well as \cite{BoBMKaMe-16} for more combinatorial techniques.

\section{Singularity analysis and asymptotics of the boundary distribution}
\label{sec:asymptotics}

The asymptotics (up to a constant) of the boundary measures has been obtained by Dai and Miyazawa in \cite{DaMi-13}, see also \cite{dai_reflecting_2011,franceschi_asymptotic_2016} for the interior measure. In this section we show that our expression \eqref{eq:main_formula_with_constants} for the Laplace transform $\phi_1$ stated in Theorem \ref{thm:main} is perfectly suited for singularity analysis, and accordingly to study the asymptotics (including the computation of the constant) of the boundary stationary distributions $\nu_1$ and $\nu_2$. We shall first recall the result of \cite{DaMi-13} and express it in terms of our notations. Then we will explain how, thanks to Theorem \ref{thm:main}, we could obtain a new proof of this result and make the constants explicit.

\subsection{Asymptotic results}
\label{sec:asymptotic_results}

\begin{thm}[Theorem 6.1 of \cite{DaMi-13}]
\label{thm:asymptDai}
\textcolor{black}{Under the assumption \eqref{eq:drift_negative},} the following asymptotics holds:
\begin{equation}
\label{eq:asymptDai}
\lim_{x\to\infty} \frac{\nu_1 (x)}{x^\kappa e^{-\tau_2 x}} = b,
\end{equation}
where $\kappa\in\{-\frac{3}{2},-\frac{1}{2},0,1\}$ and $\tau_2\in\{p,p',\theta_2^+\}$ are given in Table \ref{tableau}, and $b$ is some positive constant. See Figures \ref{fig:cases1} and \ref{fig:cases2} to visualize geometrically the different cases.
\end{thm}

We now propose a series of three remarks on Theorem \ref{thm:asymptDai}.

\smallskip

$\bullet$ We have already introduced $p$, see \eqref{def:definition_p}, and $p'$ is as follows: it is the (unique, when it exists) non-zero point 
$p'=\Theta_2^+(r)$, with $r$ defined by $\gamma_2(r,\Theta_2^-(r))=0$ and $r \leqslant \Theta_1^\pm(\theta_2^+)$.
It will be convenient to adopt the following notation: we will write $p> \theta_2^+$ (resp.\ $p'> \theta_2^+$) when $p$ (resp.\ $p'$) does not exist. 

\smallskip

$\bullet$ Theorem 6.1 of \cite{DaMi-13} deals with the asymptotics of $\nu_1(x,\infty)$, and not with that of $\nu_1(x)$, as stated in Theorem \ref{thm:asymptDai} below. However, the two statements are equivalent: when the asymptotics of the density has the form $bx^\kappa e^{-\tau_2 x}$ \textcolor{black}{as in \eqref{eq:asymptDai}}, the corresponding tail probability is given by the exact same asymptotics (with another constant $b$), see \cite[Lemma D.5]{dai_reflecting_2011}.

\smallskip

$\bullet$ Before stating Table \ref{tableau}, which gives the values of $\kappa$ and $\tau_2$ of Theorem \ref{thm:asymptDai}, we briefly recall Dai and Miyazawa's notations: $\theta^{(2,\max)}$ is the (unique) point of the ellipse $\gamma=0$ such that $\theta_2^{(2,\max)}=\theta_2^+$,
$\tau_2= \sup\{ \theta_2 >0 : \exists \theta_1\in\mathbb{R}, \phi(\theta_1,\theta_2) < \infty  \}$ and $\theta^{(2,r)}$ is the intersection point of the straight line $\gamma_1=0$ and the ellipse $\gamma=0$. Notice that the definition of $\tau_2$ does not rely on an analytic continuation of $\phi$.

\begin{table}[hbtp] 
\centering
\begin{tabular}{|l|l|l|l|r|l|}
  \hline
  \multicolumn{2}{|c|}{Cases} & \multicolumn{2}{c|}{Dai and Miyazawa's categories} & $\kappa$ & $\tau_2$  \\
  \hline
%  1. \  
  \multirow{4}{*}{$p$ or $p'\in (0,\theta_2^+)$ } \hfill1.a\hspace{0.25mm} & $\displaystyle p < \Theta_2^-(\theta_1^-)^{\phantom{'}}$ & \multirow{4}{*}{$\tau_2 < \theta_2^{(2,\max )}$} & Categories I or I\!I & $0$& $p$  \\
  \cline{2-2} \cline{4-6}
   \hfill1.b &  $\Theta_2^-(\theta_1^-) \leqslant p<{p'}^{\phantom{'}}$  &  & Category I & $0$& $p$
    \\
  \cline{2-2} \cline{4-6}
\hfill1.c\hspace{0.45mm}  &  $\Theta_2^-(\theta_1^-) \leqslant p'<p$  &  & Category I\!I\!I, $\tau_2 \neq {\theta_2^{(2,r)}}^{\phantom{'}}$ &    $0$& $p'$ 
   \\
  \cline{2-2} \cline{4-6}
 \hfill1.d  &  $\Theta_2^-(\theta_1^-) \leqslant p=p'$ &  & Category I\!I\!I, $\tau_2 = \theta_2^{(2,r)}$ &  $1$& $p$  \\ \hline  
%2. \qquad 
\multirow{4}{*}{$p$ and $p' \geqslant \theta_2^+$   } \hfill2.a\hspace{0.25mm} & $p$ and $p'>\theta_2^+$  & \multirow{4}{*}{$\tau_2 = \theta_2^{(2,\max )}$} &
Category I, $\theta^{(2,r)} \neq {\theta^{(2,\max)}}^{\phantom{'}}$
 & $-\frac{3}{2}$ & $\theta_2^+$

\\  \cline{2-2} \cline{4-6}
\hfill2.b   & $\theta_2^+=p^{\phantom{'}}$ &  & Category I, $\theta^{(2,r)} = {\theta^{(2,\max)}}^{\phantom{'}}$ &  $-\frac{1}{2}$& $\theta_2^+$  \\
  \cline{2-2} \cline{4-6}
\hfill2.c\hspace{0.45mm}  & $\theta_2^+={p'}^{\phantom{'}}$ &  & Category I\!I, $\theta^{(2,r)} \neq {\theta^{(2,\max)}}^{\phantom{'}}$ & $-\frac{1}{2}$& $\theta_2^+$   \\
  \cline{2-2} \cline{4-6}
 \hfill2.d & $\theta_2^+=p={p'}^{\phantom{'}}$ &  & Category I\!I, $\theta^{(2,r)} = {\theta^{(2,\max)}}^{\phantom{'}}$ &  $0$& $\theta_2^+$ \\
  \hline
\end{tabular}
\caption{Dai and Miyazawa's categories expressed with our notations}
\label{tableau}
\end{table}

\begin{figure}[hbtp]
\centering
\includegraphics[scale=0.6]{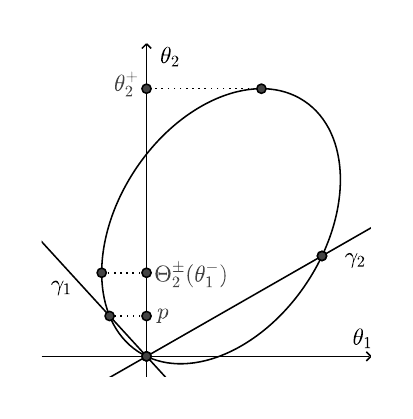}
\hspace{-1.1cm}
\includegraphics[scale=0.9]{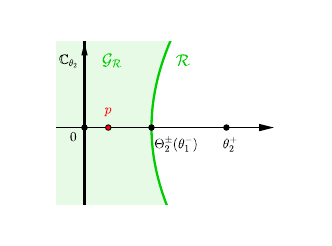}
\hspace{-1.1cm}
\includegraphics[scale=0.6]{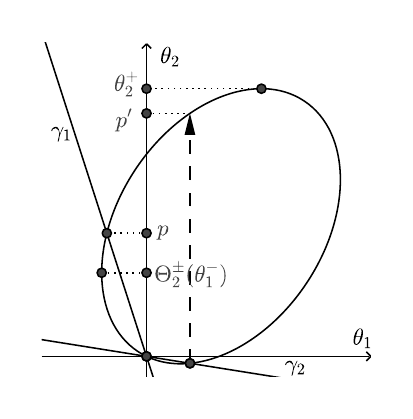}
\hspace{-1.1cm}
\includegraphics[scale=0.9]{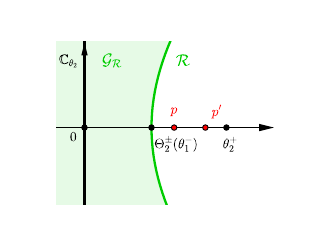}
\hspace{-1.1cm}
\includegraphics[scale=0.6]{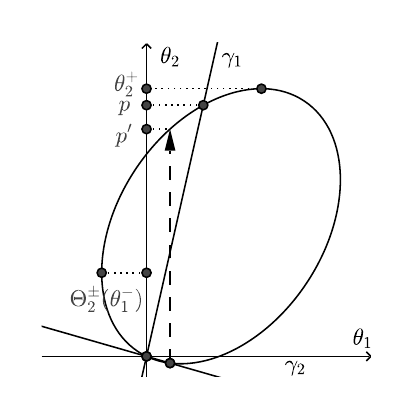}
\hspace{-1.1cm}
\includegraphics[scale=0.9]{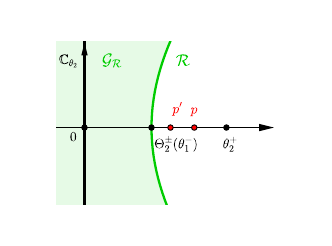}
\hspace{-1.1cm}
\includegraphics[scale=0.6]{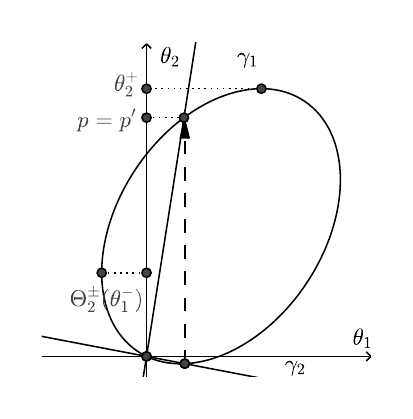}
\hspace{-1.1cm}
\includegraphics[scale=0.9]{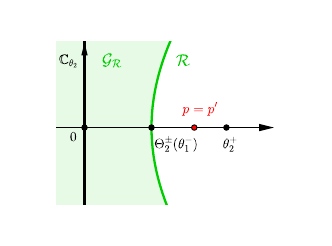}
\caption{Ellipses and curves in the cases 1.a, 1.b, 1.c and 1.d of Table \ref{tableau}}
\label{fig:cases1}
\end{figure}

%0.8 1.3

\begin{figure}[hbtp]
\centering
\includegraphics[scale=0.6]{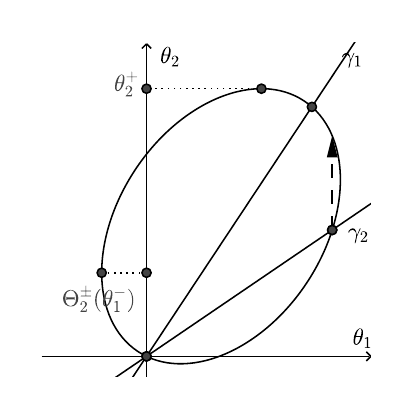}
\hspace{-0.8cm}
\includegraphics[scale=0.6]{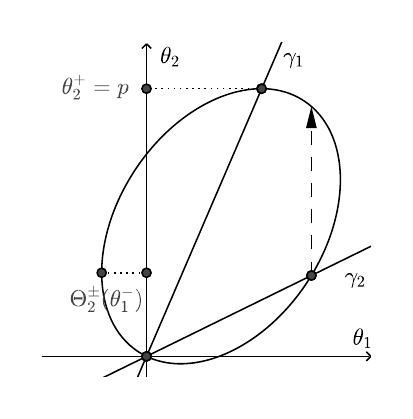}
\hspace{-0.8cm}
\includegraphics[scale=0.6]{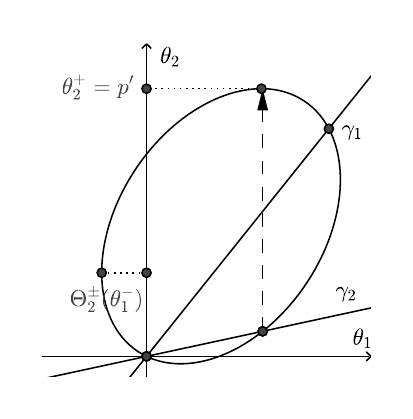}
\hspace{-0.8cm}
\includegraphics[scale=0.6]{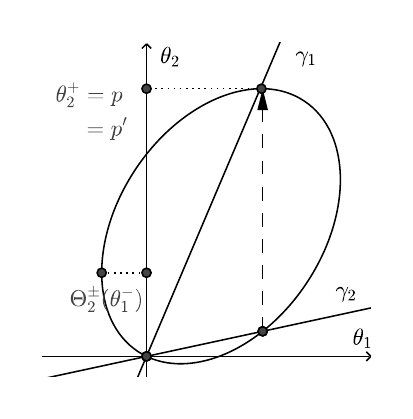}
\caption{Ellipses in the cases 2.a, 2.b, 2.c and 2.d of Table \ref{tableau}}
\label{fig:cases2}
\end{figure}

\subsection{Sketch of the proof of Theorem \ref{thm:asymptDai}}

The idea is to study the singularities of $\phi_1$ and to use transfer theorems (such as \cite[Theorem 37.1]{doetsch_introduction_1974}) relating the asymptotics of a function and the singularities of its Laplace transform. In our case, the singularity closest to $0$ will determine the asymptotics. Our aim here is not to propose a complete proof of Theorem \ref{thm:asymptDai} (for this we refer to \cite{DaMi-13}), but rather to illustrate that Theorem \ref{thm:main} indeed implies Theorem \ref{thm:asymptDai}, with the additional feature of providing an exact expression for the constant $b$ in \eqref{eq:asymptDai}. We give details for the case~1.a of Table \ref{tableau}, and only sketch the difficulties that would arise in the remaining cases.

\begin{proof}[Proof of Theorem \ref{thm:asymptDai} and computation of $b$ in case 1.a]
In that case, the index $\chi$ equals $-1$ and the formula \eqref{eq:main_formula_with_constants} gives
\begin{equation*}
\phi_1(\theta_2)
=\\\nu_1(\mathbb{R}_+)  \left( \frac{w(0)-w(p)}{w(\theta_2)-w(p)} \right) \exp\bigg\{\frac{1}{2i\pi} \int_{\R^-} \log G(\theta) \left[ \frac{w'(\theta)}{w(\theta)-w(\theta_2)}
- \frac{w'(\theta)}{w(\theta)-w(0)}
\right]
\mathrm{d}\theta\bigg\}.
\end{equation*}
Recall that $w$ is analytic, one-to-one on $\G$ and further satisfies $w(\theta_2)\neq w(\theta)$, for all $\theta_2\in\G$ and $\theta\in\R^-$. As a first result, the integral part (thus also its exponential) is analytic in the domain $\G$. A second consequence is that $\frac{1}{w(\theta_2)-w(p)}$ has a simple pole at $p$:
$
\phi_1(\theta_2)
=\frac{b+o(1)}{\theta_2-p},
$
with
\begin{equation*}
b= \nu_1(\mathbb{R}_+)  \left( \frac{w(0)-w(p)}{w'(p)} \right) \exp\bigg\{\frac{1}{2i\pi} \int_{\R^-} \log G(\theta) \left[ \frac{w'(\theta)}{w(\theta)-w(p)}
- \frac{w'(\theta)}{w(\theta)-w(0)}
\right]
\mathrm{d}\theta\bigg\}.
\end{equation*}
Theorem 37.1 of \cite{doetsch_introduction_1974} gives the announced asymptotics $\nu_1(x)=e^{-px}(b+o(1))$ as $x\to\infty$.
\end{proof}

In the other cases, the singularities are not in $\G$ and we thus need to extend meromorphically $\phi_1$ in a larger domain:

\begin{prop}[Theorem 11 of \cite{franceschi_asymptotic_2016}]
\label{prop:continuation_Laplace_transform}
The Laplace transform $\phi_1$ can be meromorphically continued on the domain $\mathbb C\setminus [\theta_2^+,\infty)$.
\end{prop}

\begin{comment}
We may notice that when $\gamma_1(\theta_1^-,\Theta_2^-(\theta_1^-)) = 0$, $\psi_1(z)=e^{\Gamma(z)}=\frac{1}{\sqrt{z}}\Omega_0(z)$ which is not meromorphic at $0$. But $\phi_1=\psi_1(W(\theta_2))$ remain meromorphic at $\Theta_2^-(\theta_1^-)$ because $W$ and $W'$ cancel at $\Theta_2^-(\theta_1^-)$ and we find again that $\phi_1$ has a pole at this point.
\end{comment}

Proposition \ref{prop:continuation_Laplace_transform} has already been proved in \cite{franceschi_asymptotic_2016}, see Theorem 11 there. Note that the formula~\eqref{eq:main_formula_with_constants} of Theorem~\ref{thm:main} provides an alternative, direct, analytic proof, which can be sketched as follows: the equation \eqref{eq:main_formula_with_constants}
is valid a priori only for $\theta_2$ in $\G$. However, considering a Hankel contour similar to that of Figure \ref{fig:integration_contour_Tutte}, surrounding $[\theta_2^+,\infty)$, we could write $\phi_1$ as an integral over the cut $[\theta_2^+,\infty)$. The study of the so-obtained formula would give the singularities $p$, $p'$ and $\theta_2^+$ as in Table \ref{tableau}, and would lead to the precise asymptotics (we could even obtain the full asymptotic development) and the computation of $b$.

\begin{rem}
\label{rem:different_cases_asymptotic}
We have already commented on the fact \textcolor{black}{that} within a single formula, Theorem~\ref{thm:main} actually captures two different expressions, depending on the value of $\chi$ in \eqref{eq:def_chi}. From an asymptotic viewpoint, Section \ref{sec:asymptotics} shows that different cases exist as well, depending on various parameters. It should be noted that these cases are all different, i.e., ${d}$ and $\chi$ do not govern the asymptotic behavior of the boundary measures. 
\end{rem}

\subsection{Asymptotics at $0$ of the boundary densities}
\label{subsec:asymp_0}

\textcolor{black}{Let us introduce the key parameter
\begin{equation}
\label{eq:definition_alpha}
     \alpha=\frac{\epsilon+\delta-\pi}{\beta},
\end{equation}
where $\epsilon$ and $\delta$ are the reflection angles defined in Appendix \ref{app:BM_cones}. They are in $(0,\pi)$ and satisfy
\begin{equation}
\label{eq:expression_delta_epsilon2}
  \tan   \epsilon=
  \frac{\sin\beta}{\frac{r_{21}}{r_{11}}\sqrt{\frac{\sigma_{11}}{\sigma_{22}}}+\cos\beta}
  \qquad\text{and}\qquad
  \tan   \delta=
  \frac{\sin\beta}{\frac{r_{12}}{r_{22}}\sqrt{\frac{\sigma_{22}}{\sigma_{11}}}+\cos\beta}.
\end{equation}
The quantity $\alpha$ is constantly used in the literature to establish some criteria. For example, $\alpha<1$ if and only if the semimartingale reflected Brownian motion exists, see \cite{Williams-85}. See also \cite{LaReZw-16} for a recent paper on the case $\alpha\in(1,2)$.}

\textcolor{black}{As we shall prove now, the exponent in the asymptotics at $\infty$ of $\phi_1$ is directly related to this crucial parameter of the model:}
\begin{prop}%[Infinite behavior of $\phi_1$]
\label{prop:behavior_phi1_infty}
%We have with $\alpha<1$ and 
\textcolor{black}{Let $\alpha$ be defined in \eqref{eq:definition_alpha}. There exists a positive constant $C>0$ such that
\begin{equation}
\label{eq:behavior_phi1_infty}
     \phi_1(\theta_2)\underset{\theta_2\to \infty}{\sim} C \theta_2^{\alpha-1}.
\end{equation}}
\end{prop}

\textcolor{black}{Before proving \eqref{eq:behavior_phi1_infty}, let us explain informally how the previous result should be related to the asymptotics of the stationary distribution at $0$.}

\begin{rem}
\textcolor{black}{Combined to the functional equation \eqref{eq:Laplace_transform_boundary}, Proposition \ref{prop:behavior_phi1_infty} and its analogue for $\phi_2(\theta_1)$ lead to, for any $s\in[0,\infty]$,
\begin{equation}
\label{eq:behavior_phi_infty}
     \phi(\theta_1,\theta_2)\underset{\theta_1,\theta_2 \to \infty \atop \theta_2 / \theta_1 \to s}{\sim} C_s \bigl(\theta_1^2+\theta_2^2\bigr)^{\frac{\alpha-2}{2}},
\end{equation}
for some constant $C_s>0$. Moreover, classical transfer theorems connect singularities and asymptotics of a function and its (inverse) Laplace transform. The latter suggest that \eqref{eq:behavior_phi1_infty} and \eqref{eq:behavior_phi_infty} might be transferred as (below, $(r,t)$ represent the polar coordinates of $(x_1,x_2)$)
\begin{equation*}
     \nu_1(x_2) \underset{x_2\to 0}{\sim} C x_2^{-\alpha}\qquad \text{and}\qquad
     \pi(x_1,x_2)=\pi(r\cos t , r \sin t) \underset{r\to 0}{\sim} C_t r^{-\alpha}
\end{equation*}
for positive constants $C$ and $C_t$.}

\textcolor{black}{The above equation is reminiscent of the following (non-asymptotic) formula \cite{Williams-85b}
\begin{equation*}
     {\pi}(r\cos t , r \sin t)= r^{-\alpha} \sin(\delta-\alpha t )
\end{equation*}
obtained in the driftless case, for $r>0$ and $t\in[0,\beta]$, and where $\delta$ is the reflection angle defined in \eqref{eq:expression_delta_epsilon2}. It means that the density $\pi$ behaves near the corner as in the driftless case. As explained in \cite{hobson_recurrence_1993}, a probabilistic interpretation is that the behavior at the corner is determined on the small scale where the drift may be neglected compared to the Brownian part.}

\textcolor{black}{Let us finally remark that these formulas are consistent with the fact that the process reaches the boundary if $\alpha>0$ (see \cite{varadhan_brownian_1985,Williams-85b}); indeed, in this case the stationary distribution goes to infinity near the corner.}
\end{rem}

\begin{proof}[Proof of Proposition \ref{prop:behavior_phi1_infty}]
\textcolor{black}{First of all, let us recall that we have defined $\psi_1$ such that
\begin{equation*}
     \phi_1(\theta_2)=\psi_1(W(\theta_2)),
\end{equation*}
see \eqref{eq:expression_psi_1}, where for some $q\in \mathbb{R}\in\G$,
\begin{equation*}
     W(\theta_2)=\frac{w(\theta_2)+1}{w(\theta_2)-w(q)},
\end{equation*}
see \eqref{eq:definition_W}.
In the rest of the proof, $C$ will denote a non-zero constant that may differ from line to line.}

\textcolor{black}{The formulas \eqref{eq:formula_psi_1} and \eqref{eq:formula_expGamma} entail that 
\begin{equation*}
     \psi_1(t)\underset{t\to 1}{\sim} C(t-1)^{\frac{{d}+\Delta}{2\pi}-\chi}.
\end{equation*}
%which means
%$$
%\psi_1(t)\underset{t\to 1}{\sim}
%\begin{cases}
% (t-1)^{\frac{\Delta}{2\pi}} \text{ if }
%  \gamma_1(\theta_1^-,\Theta_2^\pm (\theta_1^-))<0 \ ( \text{in this case } \Delta>0, \ \chi=0), 
% \\
% (t-1)^{\frac{\Delta}{2\pi}+1} \text{ if }
%  \gamma_1(\theta_1^-,\Theta_2^\pm (\theta_1^-))>0 \ ( \text{in this case } \Delta<0, \ \chi=-1),
% \\
% (t-1)^{\frac{\Delta}{2\pi}+\frac{1}{2}} \text{ if }
% \gamma_1(\theta_1^-,\Theta_2^\pm (\theta_1^-))=0
% \ (\text{in this case } \Delta\geqslant -\pi, \ \chi=0),
% \end{cases}
%$$
Moreover, by \eqref{eq:tan_Delta} we have
\begin{equation*}
     \tan \frac{{d}+\Delta}{2}=\frac{\det R \sqrt{\det \Sigma}}{\sigma_{12}(r_{11}r_{22}+r_{12}r_{21})-\sigma_{22}r_{11}r_{12}-\sigma_{11}r_{22}r_{21}},
\end{equation*}
and by \eqref{eq:def_index}, $\chi=\lfloor \frac{{d}+\Delta}{2\pi}\rfloor$.
%${d}=0$ if $\gamma_1(\theta_1^-,\Theta_2^\pm (\theta_1^-))\neq 0$ and ${d}=\pi$ if $\gamma_1(\theta_1^-,\Theta_2^\pm (\theta_1^-))=0$.
Further, by \eqref{eq:definition_w} and \eqref{eq:Chebyshev_polynomial}, $w(\theta_2)\underset{\theta_2\to\infty}{\sim} C {\theta_2^{\frac{\pi}{\beta}}}$, thus $W(\theta_2)-1 \underset{\theta_2\to\infty}{\sim}  C{\theta_2^{-\frac{\pi}{\beta}}}$. We deduce that
\begin{equation*}
     \phi_1(\theta_2)\underset{\theta_2\to \infty}{\sim} C\theta_2^{-\frac{{d}+\Delta}{2\beta}+\chi\frac{\pi}{\beta}}.
\end{equation*}
%which means that
%$$
%\phi_1(\theta_2)\underset{\theta_2\to \infty}{\sim}
%\begin{cases}
%{\theta_2^{-\frac{\Delta}{2\beta}}} \text{ if }
%  \gamma_1(\theta_1^-,\Theta_2^\pm (\theta_1^-))<0 \ ( \text{in this case } \Delta>0, \ \chi=0), 
% \\
% {\theta_2^{-\frac{\Delta}{2\beta}-\frac{\pi}{\beta}}} \text{ if }
%  \gamma_1(\theta_1^-,\Theta_2^\pm (\theta_1^-))>0 \ ( \text{in this case } \Delta<0, \ \chi=-1),
% \\
% {\theta_2^{-\frac{\Delta}{2\beta}-\frac{\pi}{2\beta}}} \text{ if }
% \gamma_1(\theta_1^-,\Theta_2^\pm (\theta_1^-))=0
% \ (\text{in this case } \Delta\geqslant -\pi, \ \chi=0).
% \end{cases}
%$$
Using now \eqref{eq:expression_delta_epsilon} and \eqref{eq:definition_beta} gives that
\begin{equation*}
     \tan (\epsilon+\delta-\beta) =-\tan  \frac{{d}+\Delta}{2}.
\end{equation*}
Combined with the fact that $\frac{{d}+\Delta}{2}\in(-\pi,\pi)$ and $0<\epsilon+\delta-\beta<\pi$ (due to the recurrence condition and $\alpha<1$), this implies
\begin{equation*}
-\frac{{d}+\Delta}{2}=
\begin{cases}
\epsilon+\delta-\beta-\pi & \text{ if } {d}+\Delta>0, \text{ i.e., if } \chi=0,
\\
\epsilon+\delta-\beta & \text{ if } {d}+\Delta<0, \text{ i.e., if } \chi=-1,
\end{cases}
\end{equation*}
which means
\begin{equation*}
     -\frac{{d}+\Delta}{2}=\epsilon+\delta-\beta-\pi-\chi \pi.
\end{equation*}
We deduce that
\begin{equation*}
     \phi_1(\theta_2)\underset{\theta_2\to \infty}{\sim} C\theta_2^{\alpha-1}.\qedhere
\end{equation*}}
\end{proof}

\section{Algebraic nature and simplification of the Laplace transforms}
\label{sec:algebraic_nature}

\subsection*{\textcolor{black}{Motivations}} In this section we are interested in the following question: in which extent is it possible to simplify the expressions of the Laplace transforms given in Theorem \ref{thm:main}? For instance, is this possible that these functions be algebraic or even rational?

This is of paramount importance: first, simplified expressions would lead to an easier analysis, in particular for asymptotic analysis or for taking inverse Laplace transforms; second, understanding the parameters $(\Sigma,\mu, R)$ for which the Laplace transforms are rational should reveal intrinsic structure of the model. In the particular case of the identity covariance matrix $\Sigma$, some attempts of simplifications may be found in \cite[Chapter~4]{foddy_1984}.

In the literature, this question has received much interest in the discrete setting. One can first think at the famous Jackson's networks \cite{Ja-57} and their product form solutions. In a closer context, Latouche and Miyazawa \cite{LaMi-14}, Chen, Boucherie and Goseling \cite{ChBoGo-15}, obtain geometric necessary and sufficient conditions for the stationary distribution of random walks in the quarter plane to be sums of geometric terms. Such criteria can be applied, e.g., to derive an approximation scheme to error bounds for performance measures of random walks in the quarter plane \cite{ChBoGo-16}. 

In our context of reflected diffusions in the quadrant, analogues of these results are obtained by Dieker and Moriarty \cite{DiMo-09}: a simple condition (involving the single angle \eqref{eq:condition_DiMo-09}) for the stationary density to be a sum of exponential terms is derived. We will discuss the links with our results in Section~\ref{subsec:exceptional_product_forms}.
\textcolor{black}{In the simplest case the sum of exponential terms consist of one single term: this is the skew-symmetric condition \eqref{eq:skew-symmetric} of \cite{harrison_multidimensional_1987}. In this case, a factorization such as in Remark \ref{rem:decoupling} exists and applying an invariant method we are able to solve the skew-symmetric case, see our Section \ref{subsec:skew}, yielding new proofs to already known results.}

\textcolor{black}{Another main reason which can lead to simplified expressions comes from the rationality of $\frac{\beta}{\pi}$, with $\beta$ in \eqref{eq:definition_beta}. Before being more precise, let us mention that this reason is deeply different than the first one: $\beta$ only depends on the covariance matrix and is therefore independent of the reflection matrix, while Dieker and Moriarty's condition is also dependent on the reflection angles, see \eqref{eq:condition_DiMo-09}. In the discrete setting, the rationality of $\frac{\beta}{\pi}$ is rather interpreted as a condition on the finiteness of a certain group of transformations, and this finiteness} was shown to have a decisive influence on the D-finiteness (a function is D-finite if it satisfies a linear differential equation with polynomial coefficients on $\mathbb Q$) of the generating functions, see \cite{fayolle_random_1999,BMMi-10,Dreyfus}. For reflected Brownian motion in the quadrant with orthogonal reflections, it is shown in \cite{franceschi_tuttes_2016} that the Laplace transform is algebraic (note, any algebraic function is D-finite) if and only if the group is finite. We present a structural result in Section \ref{subsec:group_product_forms}.

%In Section \ref{subsec:factorization_product_forms} we prove a criterion for the algebraicity of the Laplace transform, peculiar to the finite group case. Namely, we show that $\phi_1$ is algebraic if and only if a certain norm is equal to $1$, which in turn can be reformulated as a condition on a single angle, this way obtaining a result close to Dieker and Moriarty's condition.

Finally, we focus in Section \ref{sec:orthogonal_reflection} on the case of orthogonal reflections, and derive a new proof of the main result of \cite{franceschi_tuttes_2016}, as a consequence of Theorem \ref{thm:main}.

\subsection{Dieker and Moriarty's criterion}
\label{subsec:exceptional_product_forms}

For the sake of completeness, let us mention the following result:
\begin{thm}[Theorem 1 in \cite{DiMo-09}]
\label{thm:DiMo-09}
The stationary density is a sum of exponentials if and only if
\begin{equation}
\label{eq:condition_DiMo-09}
   \textcolor{black}{ \alpha= \frac{\epsilon+\delta-\pi}{\beta}}
    \in -\mathbb N=-\{0,1,2,\ldots \}, 
\end{equation}
with $\epsilon$ and $\delta$ in $(0,\pi)$ and
\begin{equation}
\label{eq:expression_delta_epsilon}
  \tan   \epsilon=
  \frac{\sin\beta}{\frac{r_{21}}{r_{11}}\sqrt{\frac{\sigma_{11}}{\sigma_{22}}}+\cos\beta},
  \qquad
  \tan   \delta=
  \frac{\sin\beta}{\frac{r_{12}}{r_{22}}\sqrt{\frac{\sigma_{22}}{\sigma_{11}}}+\cos\beta}.
\end{equation}
\end{thm}

Notice that this is not the exact statement of \cite[Theorem 1]{DiMo-09}, as in Dieker and Moriarty's paper, the Brownian motion is assumed to have an identity covariance matrix and evolves in a wedge with an arbitrary opening angle, whereas we consider Brownian motion with arbitrary covariance matrix but in the quarter plane. A simple linear transform, which is made explicit in Appendix \ref{app:BM_cones}, makes both statements equivalent. The expression \eqref{eq:expression_delta_epsilon} of the angles $\epsilon$ and $\delta$ follows from this transformation.

\subsection{Skew-symmetric case}
\label{subsec:skew}

\textcolor{black}{
The skew-symmetric case holds when the matrix condition \eqref{eq:skew-symmetric} is satisfied \cite{harrison_multidimensional_1987}. In dimension two, \eqref{eq:skew-symmetric} can be reduced to the single equation
\begin{equation}
\label{eq:skewconditiondim2}
     2\sigma_{12}=\frac{r_{21}}{r_{11}}\sigma_{11}+\frac{r_{12}}{r_{22}}\sigma_{22}.
\end{equation}
Using the identities in \eqref{eq:definition_angles_delta_epsilon} we easily show that the above condition is equivalent to
\begin{equation*}
    \epsilon+\delta=\pi,
\end{equation*}
which is the case where \textcolor{black}{the quantity $\alpha$} of Dieker and Moriarty's criterion is equal to $0$, see \eqref{eq:condition_DiMo-09}.}

\textcolor{black}{It is known, see \cite{harrison_multidimensional_1987}, that condition \eqref{eq:skew-symmetric} is satisfied if and only if the stationary distribution has a product form, i.e., $\pi(x_1,x_2)=\pi_2(x_1)\pi_1(x_2)$, where the $\pi_i$'s are the marginal densities of $\pi$. Furthermore it implies that the stationary distribution is exponential, meaning that 
\begin{equation}
\label{eq:product_form}
     \pi(x_1,x_2)=\zeta_1\zeta_2e^{-\zeta_1 x_1-\zeta_2 x_2},\quad  \text{with } \small\left(  \begin{array}{l} \zeta_1 \\  \zeta_2  \end{array} \right)=-2 \cdot\text{diag} (\Sigma)^{-1}\cdot\text{diag}(R)\cdot R^{-1}\cdot \mu.
\end{equation}
See for example \cite[\S 10]{HaRe-81}, \cite{harrison_multidimensional_1987} or \cite{DaMi-13} for more details on these results.}

\textcolor{black}{In fact, in the skew-symmetric case, a rational factorization 
%$$G(\theta_2)=\frac{F(\theta_2)}{F(\overline{\theta_2})},$$ 
such as in the Remark \ref{rem:decoupling} exists. It is then possible to find again, in another way, some already known results.
Indeed if the skew-symmetric condition holds, we shall prove below that for $\theta_2\in\R$,
\begin{equation}
\label{eq:factorization_ss}
     G(\theta_2)=\frac{F(\theta_2)}{F(\overline{\theta_2})},\quad \text{with } F(\theta_2)=\zeta_2 -\theta_2.
\end{equation}
where by \eqref{eq:product_form} $\zeta_2$ takes the value
\begin{equation}
\label{eq:value_alpha_2}
     \zeta_2=\frac{2r_{22}(r_{21}\mu_1-r_{11}\mu_2)}{\sigma_{22}\det R}.
\end{equation}     
}

\textcolor{black}{Using \eqref{eq:skewconditiondim2} and \eqref{eq:expression_p} it is easy to remark that $\zeta_2=p$, which is the only possible pole of $\phi_1$ in $\G$ as we observe after Lemma \ref{lem:continuation}. It follows from Proposition \ref{prop:BVP_Carleman} that the function $\phi_1(\theta_2) (\zeta_2-\theta_2)$
\begin{itemize}
     \item is bounded on $\G$ and converges at infinity to $\nu_1(0)$ (thanks to the initial value theorem and Lemma \ref{lem:continuation}),
     \item is continuous on $\overline{\G}$,
     \item satisfies the boundary condition $\phi_1(\theta_2) (\zeta_2-\theta_2)=\phi_1(\overline{\theta_2}) (\zeta_2-\overline{\theta_2})$ for all $\theta_2\in\R$.
\end{itemize}
Then, using an invariant lemma (see Lemma 2 in \cite[Section 10.2]{litvinchuk_solvability_2000}), we conclude that for some constant $C$, 
\begin{equation*}
     \phi_1(\theta_2)=\frac{C}{\zeta_2 -\theta_2}.
\end{equation*}     
Evaluating the above equation at $0$ and using Lemma \ref{lem:value_at_0} gives $C=\frac{\sigma_{11}}{2r_{11}}\zeta_1\zeta_2$.
Inverting the Laplace transform implies that the stationary distribution is exponential. Then using the functional equation \eqref{eq:functional_equation}, we find that $\frac{\sigma_{11}\sigma_{22}}{4r_{11}r_{22}}\zeta_1\zeta_2 \phi(\theta_1,\theta_2)=\phi_1(\theta_2)\phi_2(\theta_1)$, which means that $\pi$ has a product form.}
\textcolor{black}{\begin{proof}[Proof of Equation \eqref{eq:factorization_ss}]
Let $\theta_2\in\mathcal{R}$ and note $\theta_1=\Theta_1^-(\theta_2)$. Elementary computations give 
\begin{equation*}
     \gamma_1(\theta_1,\theta_2)\gamma_2(\theta_1,\overline{\theta_2})=\theta_1\left(\frac{2 r_{22}}{\sigma_{22}}(r_{21}\mu_1-r_{11}\mu_2)-\theta_2\det R\right). 
\end{equation*}
To find this, we just have to use the skew-symmetric condition \eqref{eq:skewconditiondim2} and to remark that for $\theta_2\in\R$ Vieta's formulas give $\theta_2 \overline{\theta_2}=\frac{1}{\sigma_22}(\sigma_{11}\theta_1^2+2\mu_1\theta_1)$ and $\overline{\theta_2}=-\theta_2 -\frac{1}{\sigma_{22}}(2\sigma_{12}\theta_1+2\mu_2)$. 
Then, using the expression \eqref{eq:value_alpha_2} of $\zeta_2$ we find $G(\theta_2)=\frac{\zeta_2-\theta_2}{\zeta_2-\overline{\theta_2}}$ for $\theta_2\in\R$.
\end{proof}
}

\subsection{Structural form of the Laplace transforms}
\label{subsec:group_product_forms}

In the case $\frac{\beta}{\pi}\in\mathbb Q$, and in this case only, the function $W$ in \eqref{eq:definition_W} is algebraic (as the generalized Chebyshev polynomial \eqref{eq:definition_w} is, see Remark \ref{eq:algebraic_nature_w_W}), yielding the following structural result:

\begin{prop}
\label{prop:structural_group_finite}
If $\frac{\beta}{\pi}\in\mathbb Q$, the Laplace transform $\phi_1$ of Theorem \ref{thm:main} is the product of an algebraic function by the exponential of a D-finite function.
\end{prop}

\begin{proof}
This easily follows from the fact that the Cauchy integral of a D-finite function is D-finite, see, e.g., \cite{Ta-92}.
\end{proof}
However, it is not true in general that the exponential of a D-finite function is still D-finite.

\subsection{Orthogonal reflections}
\label{sec:orthogonal_reflection}

Here we consider the case of orthogonal reflections, which is equivalent for the reflection matrix $R$ in \eqref{eq:RBMQP} to be the identity matrix; see also Figure \ref{fig:drift_reflection}. By developing the theory of Tutte's invariants (introduced by Tutte in \cite{Tutte-95} for the enumeration of properly colored triangulations, and used in \cite{bernardi_counting_2015} for the enumeration of quadrant walks) for the Brownian motion, we proved in \cite{franceschi_tuttes_2016} the following result:
\begin{thm}[Theorem 1 in \cite{franceschi_tuttes_2016}]
\label{thm:Tutte_ESAIM}
Let $R$ be the identity matrix in~\eqref{eq:RBMQP}. The Laplace transform $\phi_1$ is equal to 
\begin{equation} 
\label{eq:worth}
     \phi_1(\theta_2)= \frac{-\mu_1 w'(0)}{w(\theta_2)-w(0)} \theta_2.
\end{equation}
\end{thm}
In this section we derive a new proof of this result, as a consequence of Theorem \ref{thm:main}. More generally, the proof below would work for any parameters $(\Sigma,\mu,R)$ such that $G(\theta_2)=\frac{F(\theta_2)}{F(\overline{\theta_2})}$ (cf.\ Remark \ref{rem:decoupling}), yielding a rational expression of $\phi_1(\theta_2)$ in terms of $w(\theta_2)$ and $\theta_2$.
%\textcolor{black}{\begin{equation*} 
%     \phi_1(\theta_2)= \frac{ w'(0)}{w(\theta_2)-w(0)}\frac{ w'(p)}{(w(\theta_2)-w(p))^{-2\chi}} \frac{1}{F(\theta_2)}.
%\end{equation*}}

\begin{proof}
Let us first notice that the index $\chi=0$. Indeed, since $G(\theta_2)=\frac{\overline{\theta_2}}{\theta_2}$, we have $\arg G(\theta_2) =-2\arg \theta_2 >0$ for $\theta_2\in \R$, and thus $\Delta>0$, see the proof of Lemma \ref{lem:index}.

Starting from the formula \eqref{eq:Wmainformula}, we have for some constant $C$
\begin{equation*}
\phi_1(\theta_2) =C\exp\bigg\{{\frac{1}{2i\pi} \int_{\R^-} (\log\overline{\theta}-\log\theta)\frac{W'(\theta)}{W(\theta)-W(\theta_2)} \mathrm{d}\theta}\bigg\}
=C\exp\bigg\{{\frac{1}{2i\pi} \int_{\R} \log\theta\frac{W'(\theta)}{W(\theta)-W(\theta_2)} \mathrm{d}\theta}\bigg\}.
\end{equation*}
To compute the above integral, we first integrate on the contour represented on Figure \ref{fig:integration_contour_Tutte}. The residue theorem gives
\begin{equation} 
\label{eq:residue}
      \frac{1}{2i\pi}\bigg\{ \int_{\R_R}+\int_{\mathcal{C}_R}+\int_{\mathcal{C}_\epsilon}+\int_{-R+i\epsilon}^{i\epsilon}+\int_{-i\epsilon}^{-R-i\epsilon}  \bigg\}\log \theta
\frac{W'(\theta)}{W(\theta)-W(\theta_2)} \mathrm{d}\theta =\log \theta_2 -\log q.
\end{equation}
\begin{figure}[t!]
\vspace{-5mm}
\centering
\includegraphics[scale=0.85]{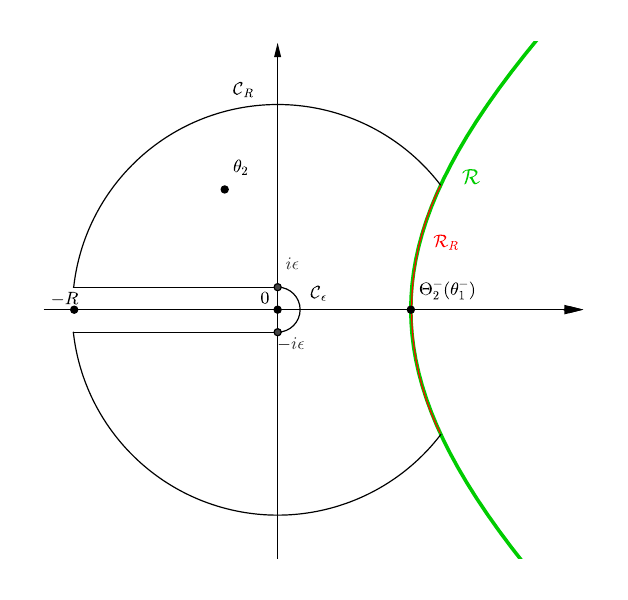}
\vspace{-5mm}
\caption{Integration contour used in the proof of Theorem \ref{thm:Tutte_ESAIM}}
\label{fig:integration_contour_Tutte}
\end{figure}
It is easy to see that 
\begin{equation*}
     \lim_{R\to\infty}\frac{1}{2i\pi}\int_{\R_R} \log \theta\frac{W'(\theta)}{W(\theta)-W(\theta_2)} \mathrm{d}\theta=\frac{1}{2i\pi}\int_\R \log \theta\frac{W'(\theta)}{W(\theta)-W(\theta_2)}\mathrm{d}\theta
\end{equation*}
and that in the limits when $\epsilon\to0$ and $R\to\infty$, the contributions on $\mathcal{C}_\epsilon$ and $\mathcal{C}_R$ both converge to $0$, because $W$ is analytic at $0$ and $\infty$, respectively. 

Furthermore,
\begin{align*}
\lim_{\epsilon\to 0} \lim_{R\to\infty}\frac{1}{2i\pi}
\bigg\{ \int_{-R+i\epsilon}^{i\epsilon}+\int_{-i\epsilon}^{-R-i\epsilon} \bigg\}&
\log \theta\frac{W'(\theta)}{W(\theta)-W(\theta_2)} \mathrm{d}\theta\\
&=\lim_{\epsilon\to 0} \frac{1}{2i\pi}\int_{-\infty}^0
(\log (t+i\epsilon)-\log (t-i\epsilon))
\frac{W'(\theta)}{W(\theta)-W(\theta_2)} \mathrm{d}\theta
\\ &=
\frac{1}{2i\pi}\int_{-\infty}^0
2i\pi\frac{W'(\theta)}{W(\theta)-W(\theta_2)} \mathrm{d}\theta
\\ &= \log \frac{W(0)-W(\theta_2)}{W(\infty)-W(\theta_2)}
\\&= \log \frac{w(\theta_2)-w(0)}{w(q)-w(0)}.
\end{align*}
Above we have used the dominated convergence and the fact that the principal determination of the logarithm gives us $\lim_{\epsilon\to 0} (\log (t+i\epsilon)-\log (t-i\epsilon)) = 2i\pi$.
Letting $R\to \infty$ and then $\epsilon\to 0$ in \eqref{eq:residue}, we have, for some constants $C$ and $C'$,
\begin{equation*}
     \phi_1(\theta_2)=C\exp\bigg\{\frac{1}{2i\pi} \int_\R \log\theta \frac{W'(\theta)}{W(\theta)-W(\theta_2)} \mathrm{d}\theta \bigg\}= C' \frac{\theta_2}{w(\theta_2)-w(0)}.
\end{equation*}
Since by Lemma \ref{lem:value_at_0} one has $\phi_1(0)=-\mu_1$, this eventually gives the right constant in \eqref{eq:worth}. 
\end{proof}
%We remark that $\frac{W'(\theta)}{W(\theta)-W(\theta_2)}=\frac{w'(\theta)}{w(\theta)-w(\theta_2)}-\frac{w'(\theta)}{w(\theta)-w(p)}$.

%
%We set
%$$
%X^+(t)=(t-1)e^{\Gamma^+(t)}
%$$
%where for $t\in[0,1]$
%$$
%\Gamma^+(t)=\frac{1}{2}\log H(t)+\frac{1}{2i\pi}\int_0^1 \frac{\log(H(s))}{s-t}\mathrm{d}s
%$$

\appendix

\section{Equivalence between Brownian motion in wedges and Brownian motion in the quarter plane}
\label{app:BM_cones}

We use the notation of Section \ref{sec:introduction}. Up to an isomorphism, studying Brownian motion in the quarter plane with arbitrary covariance matrix $\Sigma$ is equivalent to studying Brownian motion in a cone of angle $\beta = \arccos-\frac{\sigma_{12}}{\sqrt{\sigma_{11}\sigma_{22}}}$, with covariance identity. \textcolor{black}{See for example \cite[Equation (23)]{AsIaMe-96} and \cite[Lemma 3.23]{Sa-15}.} In this short section we relate the key parameters (angles of the reflection vectors and drift) before and after the linear transformation.

%To simplifie the notations we introduce $r=\frac{\sigma_{12}}{\sqrt{\sigma_{11}\sigma_{22}}}=-\cos \beta$.
Let us define the linear transforms
\begin{equation}
\label{eq:linear_transforms_T}
     T= 
\begin{pmatrix}
   \frac{1}{\sin \beta}& \cot \beta \\
   0 & 1
\end{pmatrix}
\begin{pmatrix}
   \frac{1}{\sqrt{\sigma_{11}}} & 0 \\
   0 &\frac{1}{\sqrt{\sigma_{22}}}
\end{pmatrix},\qquad
T^{-1}= 
\begin{pmatrix}
  \sqrt{\sigma_{11}} & 0 \\
   0 &\sqrt{\sigma_{22}}
\end{pmatrix}
\begin{pmatrix}
   \sin \beta & -{\cos \beta} \\
   0 & \phantom{-}1
\end{pmatrix}.
\end{equation}

 \begin{figure}[hbtp]
 \vspace{-5mm}
 \centering
 \includegraphics[scale=1]{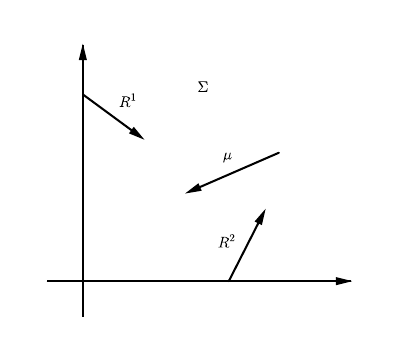}
  \includegraphics[scale=1]{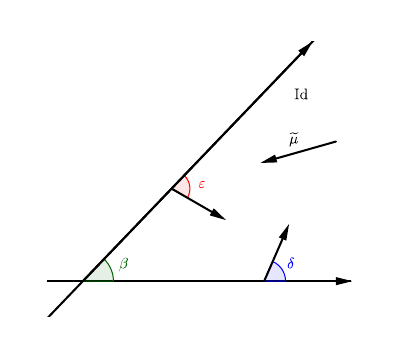}
  \vspace{-5mm}
 \caption{The linear transformation $T$ in \eqref{eq:linear_transforms_T} from the quadrant to the wedge of opening angle $\beta$}
 \label{fig:linear_transformation}
 \end{figure}
 
Obviously the reflected Brownian motion associated to $(\Sigma,\mu, R)$ becomes a Brownian motion (with covariance identity) in a wedge of angle $\beta$ and with parameters $(\text{Id},T\mu,TR)$. The new angles of reflection are $\delta$ and $\epsilon$ (cf.\ Figure \ref{fig:linear_transformation}), such that
\begin{equation}
\label{eq:definition_angles_delta_epsilon}
\left\{\begin{array}{lll}
\tan \delta =\dfrac{\sin \beta}{a +\cos \beta},  &\quad
\cos \delta = \dfrac{a+\cos \beta}{\sqrt{a^2 +2a\cos \beta+1}}, &\quad \sin \delta =  \dfrac{\sin \beta}{\sqrt{a^2 +2a\cos \beta+1}},\smallskip\smallskip\\
\tan \epsilon =\dfrac{\sin \beta}{b +\cos \beta}, &\quad
\cos \epsilon = \dfrac{b+\cos \beta}{\sqrt{b^2 +2b\cos \beta+1}}, &\quad \sin \epsilon =  \dfrac{\sin \beta}{\sqrt{b^2 +2b\cos \beta+1}},
\end{array}\right.
\end{equation}
where $a=\frac{r_{12}}{r_{22}} \sqrt{\frac{\sigma_{22}}{\sigma_{11}}}$ and $b=\frac{r_{21}}{r_{11}} \sqrt{\frac{\sigma_{11}}{\sigma_{22}}}$.
The new drift is $\widetilde{\mu}=T\mu$, where
\begin{equation*}
\widetilde{\mu}_1= \frac{\mu_1}{\sqrt{\sigma_{11}}}\frac{1}{\sin \beta}+\frac{\mu_2}{\sqrt{\sigma_{22}}} \cot \beta
\quad
\text{ and }
\quad
\widetilde{\mu}_2= \frac{\mu_2}{\sqrt{\sigma_{22}}}. %\sin \beta.
\end{equation*}
%Thanks to \eqref{eq:definition_theta_pm} and \eqref{eq:value_s_0}, we have $\tan \arg s_0 = \frac{\widetilde{\mu}_2}{\widetilde{\mu}_1}$, see Figure \ref{fig:ellipse_uniformization}. 

%\textcolor{black}{Dans cette section \'evoquer \cite[eq (23)]{AsIaMe-96} quelque part et \cite[Lemma 3.23]{Sa-15}}
\textcolor{black}{To conclude the appendix, we prove Corollary \ref{cor:main}, which gives an explicit formula for the Laplace transform of the stationary distribution of a reflected Brownian motion in a wedge.
\begin{proof}[Proof of Corollary \ref{cor:main}]
Let us arbitrarily choose $\sigma_{11}=\sigma_{22}=1$ and $\sigma_{12}=-\cos \beta$. It implies that the linear transform $T$ in \eqref{eq:linear_transforms_T} is the same as $T_1$ in the statement of Corollary \ref{cor:main}. We
consider the process
$Z=T^{-1}\widetilde{Z}$, which is a reflected Brownian motion in the quadrant of parameters $(\Sigma,\mu,R)$, with
\begin{equation*}
      \Sigma=T^{-1}\widetilde{\Sigma}(T^{-1})^\top,\quad \mu=T^{-1}\widetilde{\mu}\quad \text{and}\quad R=T^{-1}\widetilde{R}.
\end{equation*}     
Let $\Pi$ (resp.\ $\widetilde \Pi$) be the invariant measure of $Z$ (resp.\ $\widetilde Z$) and $\pi$ (resp.\ $\widetilde \pi$) its density. It is easy to notice that $$\Pi = \widetilde \Pi \circ T\quad \text{and}\quad\pi=\vert\det T\vert\widetilde{\pi} \circ T.$$
Indeed, by a fundamental property of the invariant measure,
for all $x\in\mathbb{R}_+^2$ and all measurable set $A$ in $\mathbb{R}_+^2$ we have the following limits
$$
\mathbb P_x [Z_t\in A] \underset{t\to\infty}{\longrightarrow} \Pi (A),\qquad
\mathbb P_x [Z_t\in A]=\mathbb P_x[T Z_t\in T A]=\mathbb P_{Tx} [\widetilde{Z}_t \in T A]\underset{t\to\infty}{\longrightarrow} \widetilde{\Pi} (TA).
$$
It yields $\Pi = \widetilde \Pi \circ T$. Furthermore by a simple change of variable we have
$$
\widetilde \Pi \circ T (A)
=\int_{TA}\widetilde{\pi}(\widetilde x)\mathrm{d} \widetilde x = |\det T|\int_A \widetilde{\pi} (Tx) \mathrm{d}x,
$$
therefore $|\det T|\widetilde{\pi} \circ T$ is the density of $\Pi \circ T$ and is then equal to $\pi$.
Theorem \ref{thm:main} gives the value of $\phi$, the Laplace transform of $\pi$. Lastly, a simple change of variable $\widetilde{x}=T x$ in Equation \eqref{eq:Laplace_transform_interior} leads to 
%\begin{align*}
%\widetilde{\phi}(\widetilde{\theta}) &=
%\iint_{\mathcal{C}_\beta} \exp (\widetilde{\theta} \cdot \widetilde{x}) \widetilde{\pi} (\widetilde{x}) \mathrm{d} \widetilde{x}
%\\ &=
%\iint_{T^{-1} \mathcal{C}_\beta} \exp (\widetilde{\theta} \cdot T x) \widetilde{\pi} (T x) |\det T | \mathrm{d} x
%\\ &=
%\iint_{\mathbb{R}_+^2} \exp (T^\top \widetilde{\theta} \cdot x) \pi(x) \mathrm{d} x
%\\ &= \phi(T^\top \widetilde{\theta})
%,
%\end{align*}
\begin{equation*}
\widetilde{\phi}(\widetilde{\theta}) =
\iint_{\mathcal{C}_\beta} \exp (\widetilde{\theta} \cdot \widetilde{x}) \widetilde{\pi} (\widetilde{x}) \mathrm{d} \widetilde{x}
=
\iint_{T^{-1} \mathcal{C}_\beta} \exp (\widetilde{\theta} \cdot T x) \widetilde{\pi} (T x) \vert\det T \vert \mathrm{d} x
=
\iint_{\mathbb{R}_+^2} \exp (T^\top \widetilde{\theta} \cdot x) \pi(x) \mathrm{d} x
 = \phi(T^\top \widetilde{\theta}),
\end{equation*}
where $\mathcal{C}_\beta$ is the wedge of angle $\beta$ where the process evolves.
\end{proof}}

\bibliographystyle{apalike} % Le style est mis entre accolades.

\bibliography{biblio1}
%\bibliography{biblio} % mon fichier de base de donn\'ees s'appelle bibli.bib

\end{document}